\documentclass[arxiv=true]{agn_article}
\usepackage[markall=true]{agn_all}

\usepackage{enumitem}

\newcommand{\Wiso}{W_{\textrm{\rm iso}}}

\newcommand{\Wvol}{W_{\textrm{\rm vol}}}

\newcommand{\Wex}{W_0}

\newcommand{\WADM}{W_{\mathrm{ADM}}}
\newcommand{\WADMp}{W_{\mathrm{ADM}^+}}
\newcommand{\WA}{W_{\mathrm{A}}}
\newcommand{\WS}{W_{\mathrm{S}}}
\newcommand{\WSp}{W_{\mathrm{S}^+}}
\newcommand{\ghatS}{\ghat_{\mathrm{S}}}

\newcommand{\Mbb}{\mathbb{M}}
\newcommand{\quoteesc}[1]{{\textnormal{[#1]}}}
\newcommand{\quoteescnobr}[1]{{\textnormal{#1}}}

\newcommand{\curve}{X}
\newcommand{\svmap}{\lambdahat}
\newcommand{\hunordered}{h}

\renewcommand{\lambdahat}{\widehat{\smash{\lambda}\vphantom{2}}}
\renewcommand{\lambdatilde}{\widetilde{\smash{\lambda}\vphantom{2}}}
\renewcommand{\hhat}{\widehat{\smash{h}\vphantom{2}}}

\providecommand{\citet}[2][]{\citeauthor{#2} \cite[#1]{#2}}

\graphicspath{{images/}}
\begin{document}
\newgeometry{scale=0.785}

\title{\vspace*{-.98cm}A rank-one convex, non-polyconvex isotropic function on $\GLp(2)$ with compact connected sublevel sets}
\date{\today}

\knownauthors[voss]{voss,ghiba,martin,neff}

\maketitle

\begin{abstract}
	\noindent
	According to a 2002 theorem by Cardaliaguet and Tahraoui,
	an isotropic, compact and connected subset of the group $\GLp(2)$ of invertible $2\times2$\,--\,matrices is rank-one convex if and only if it is polyconvex. In a 2005 \emph{Journal of Convex Analysis} article by Alexander~Mielke, it has been conjectured that the equivalence of rank-one convexity and polyconvexity holds for \emph{isotropic functions} on $\GLp(2)$ as well, provided their sublevel sets satisfy the corresponding requirements. We negatively answer this conjecture by giving an explicit example of a function $W\col\GLp(2)\to\R$ which is not polyconvex, but rank-one convex as well as isotropic with compact and connected sublevel sets.
\end{abstract}

\vspace*{1.4em}
\\[1.4em]
\tableofcontents
%
%
%
%
\section{A conjecture on rank-one convexity and polyconvexity}

Generalized notions of convexity play an important role in the multidimensional calculus of variations. In order to ensure the existence of minimizers for functionals of the form
\begin{align}
	I\col W^{1,p}(\Omega;\R^n)\to\R\,,\qquad \varphi\mapsto I(\varphi) = \int_{\Omega} W(\grad\varphi(x))\,\dx
\end{align}
in the Sobolev space $W^{1,p}(\Omega)$ with $\Omega\subset\R^n$ under various boundary conditions, weak lower semicontinuity (w.l.s.c.) of $I$ is usually required. While (classical) convexity and lower semicontinuity of the function $W\col\Rnn\to\R$ on the set of \nnmatrices\ are certainly sufficient for $I$ to be w.l.s.c., convexity is oftentimes too strong as requirement for practical applications. Therefore, weaker convexity properties have been introduced which in many cases suffice to ensure the weak lower semicontinuity and thus the existence of minimizers as well. In particular, Morrey \cite{morrey1952quasi} famously showed that, under appropriate conditions, $I$ is w.l.s.c.\ if and only if $W$ is \emph{quasiconvex}, i.e.\ if
\begin{equation}\label{eq:definitionQuasiconvexity}
	\int_{\Omega}W(F_0+\nabla \vartheta)\,\dx\geq \int_{\Omega}W(F_0)\,\dx=W(F_0)\cdot \abs{\Omega}
	\qquad\tforall F\in\Rnn
\end{equation}
for every bounded open set $\Omega\subset\R^n$ and all test functions $\vartheta\in W_0^{1,\infty} (\Omega;\R^n)$. Further generalizations of the convexity of functions have been subsequently introduced, with two of the most important ones being \emph{rank-one convexity}, i.e.
\begin{equation}\label{eq:definitionRankOneConvexity}
	W((1-t)F+t(F+H)) \;\leq\; (1-t)\.W(F) + t\.W(F+tH)
\end{equation}
for all $F\in\Rnn$, $t\in[0,1]$ and $H\in\Rnn$ with $\rank(H)=1$, and \emph{polyconvexity} of $W$ which, for $n=2$ and $n=3$, can be expressed as
\begin{align}
	W(F) = P(F,\det F) \qquad&\text{for a convex function }\;P\col\R^{2\times2}\times\R \;\cong\; \R^5\to\R_\infty\label{eq:definitionPolyconvexityTwoDimensional}\\[-0.7em]
\intertext{and}\notag\\[-2em]
	W(F) = P(F,\Cof F,\det F) \qquad&\text{for a convex function }\;P\col\R^{3\times3}\times\R^{3\times3}\times\R \;\cong\; \R^{19}\to\R_\infty\label{eq:definitionPolyconvexityThreeDimensional}
\end{align}
respectively, where $\Cof F$ denotes the cofactor of $F$ and $\R_\infty=\R\cup \{\infty\}$. Major interest in these two concepts stems from the observation that polyconvexity implies quasiconvexity, which in turn implies rank-one convexity \cite{Dacorogna08,rindler2018calculus}. Since quasiconvexity is difficult to verify or falsify for a given function $W$ on $\Rnn$, the availability of a sufficient and a necessary criterion has proven quite useful in the past. However, for $n\geq3$, these three notions of convexity are generally not equivalent \cite{sverak1992rank}. For $n=2$, on the other hand, it is still an open question whether rank-one convexity implies quasiconvexity \cite{morrey1952quasi,ball1987does,ball2002openProblems,conti2005rank,parry1995planar,parry2000rank,pedregal1998note}. In fact, many classes of functions on $\R^{2\times2}$ have been identified for which rank-one convexity even implies polyconvexity \cite{Ball84b,muller1999rank,agn_martin2015rank,agn_ghiba2018rank,agn_martin2019quasiconvex} and thus quasiconvexity.

\medskip

In a 2005 \emph{Journal of Convex Analysis} article \cite{mielke2005necessary}, Alexander Mielke put forward a further conjecture on the relation between rank-one convexity and polyconvexity in the planar case.

\begin{conjecture}[Mielke \cite{mielke2005necessary}]\label{conjecture:mielke}
	Let $W\col\GLp(2)\to \R_{\infty}$ be an objective, isotropic and rank-one convex function such that the sublevel sets
	\[
		S_c \colonequals \{F\in\R^{2\times2} \setvert \det F > 0\,,\; W(F) \leq c\}, \qquad c\in \R
	\]
	are connected and compact. Then $W$ is polyconvex.
\end{conjecture}

We note that although the requirement of compact sublevel sets is not explicitly included in the original statement of the conjecture \cite{mielke2005necessary}, it can be inferred from the context, especially since Mielke excludes an otherwise viable counterexample by Aubert despite the connectedness of the corresponding sublevel sets (cf.~\eqref{eq:aubertEnergy} and Appendix~\ref{appendix:aubertConnectedness}).

Conjecture \ref{conjecture:mielke}\footnote{%
	Mielke also later clarified his originally intended formulation \cite{mielkePersonalCommunication}. The original conjecture \cite{mielke2005necessary} is phrased as follows, with the additional requirement of compactness added here for clarity:
	\begin{quote}
	{\it%
		However, there is even further similarity between rank-one convexity and polyconvexity for isotropic functions which stems from the theory in \quoteescnobr{\cite{aubert1987counterexample,cardaliaguet2002equivalence1,cardaliaguet2002equivalence2}}. There it is shown that compact, connected, and isotropic subsets of $\{A\in \R^{2\times 2} \setvert \det A>0\}$ are rank-one convex if and only if they are polyconvex. So we conjecture that all isotropic, rank-one convex functions are in fact polyconvex, if all the sublevel sets $\{F\in \R^{2\times 2} \,\setvert\, \det F>0\,,\; \Phi(\Lambda(F))\leq t\}\,,\;t\in \R$, are \quoteesc{compact and} connected. The last condition rules out the famous counterexample by Aubert in \cite{aubert1987counterexample} given via
		\begin{equation*}
			\Phi(\nu)=\frac{1}{3} \left(\nu_1^4+\nu_2^4\right)+\frac{1}{2}\nu_1^2\nu_2^2-\frac{2}{3} \nu_1\nu_2 \left(\nu_1^2+\nu_2^2\right)\,.
		\end{equation*}%
	}
	\end{quote}
	Here, Mielke uses the notation $W(F)=\Phi(\Lambda(F))$ for the representation of an isotropic energy $W\col\GLp(2)\to\R$ in terms of the ordered singular values $\Lambda_1(F)\geq\Lambda_2(F)$ of $F$.
} is mainly motivated by analogous results relating the polyconvexity and rank-one convexity of sets (cf.~Proposition \ref{proposition:rankOnePolyEquivalenceCompactConnectedSetsGLp}). In Section \ref{section:counterExample}, however, we will show that the function (cf.~eq.~\eqref{eq:counterexampleDefinition})
\[
	\Wex\col\GLp(2)\to\R\,,\qquad \Wex(F) = \frac{\opnorm{F}^2}{\det F} - \log\biggl(\frac{\opnorm{F}^2}{\det F}\biggr) + \log(\det F) + \frac{1}{\det F}
	\,,
\]
where $\opnorm{\..\.}$ denotes the operator norm, is indeed rank-one convex and isotropic with compact and connected sublevel sets, but not polyconvex. We will thereby establish the following proposition.
\begin{proposition}
\label{prop:conjectureFalse}
	Conjecture \ref{conjecture:mielke} does not hold.
\end{proposition}

In the following, we will provide some of the background which originally motivated Conjecture \ref{conjecture:mielke} as well as some basic results required for demonstrating that $\Wex$ has the claimed properties.
%
%
%
%
\section{Background and preliminaries}

While the generalized notions of convexity discussed above are defined for (and applied to), first and foremost, \emph{functions} on $\Rnn$, analogous definitions are sometimes considered for \emph{sets} of matrices as well.

\begin{definition}\label{def:generalizedConvexitySets}
	Let $\Mbb\subset\Rnn$. Then $\Mbb$ is called \emph{rank-one convex} if $(1-t)F_1+tF_2\in \Mbb$ for all $t\in[0,1]$ and all $F_1,F_2\in \Mbb$ such that $\rank(F_2-F_1)= 1$. For $n=2$, the set $\Mbb\subset\R^{2\times2}$ is called \emph{polyconvex} if $\Mbb=\{F\in\R^{2\times2} \setvert P(F,\det F)\leq0\}$ for some convex function $P\col\R^{2\times2}\times\R\cong\R^5\to\R_\infty$.
\end{definition}

Of course, the definition of polyconvex sets can easily be extended \cite{Dacorogna08,szekelyhidi2006local} to the case of arbitrary dimension $n$. Note that Definition \ref{def:generalizedConvexitySets} of rank-one convex subsets is used by Cardaliaguet and Tahraoui \cite{cardaliaguet2000equivalence,cardaliaguet2002equivalence1,cardaliaguet2002equivalence2} as well as by Dacorogna \cite{Dacorogna08}, while Conti et al.\ \cite{conti2005rank} call subsets with the stated properties \emph{laminated convex} subsets of $\R^{2\times 2}$. 

The rather recent development of these notions of convexity for sets has, again, been primarily motivated by applications in the calculus of variations. In particular, it has been shown that if a compact set $\Mbb\subset\R^{2\times2}$ is polyconvex and $\OO(2)$-invariant (see Remark \ref{remark:isotropyHemitropyOfSets}), then the set
\[
	\mathcal{M} = \{ \varphi\in W^{1,\infty}(\Omega;\R^n) \setvert \grad\varphi(x)\in \Mbb \;\text{ for }x\in\Omega \;\text{ a.e.}\}
\]
is weakly-$\ast$ closed; furthermore, if $\mathcal{M}$ is weakly-$\ast$ closed, then $\Mbb$ is rank-one convex \cite{cardaliaguet2000equivalence,cardaliaguet2002equivalence1}.

In the context of nonlinear hyperelasticity, where generalized convexity properties have an especially long and rich history \cite{ball1976convexity,chelminski2006new,knowles1976failure,knowles1978failure,rosakis1994,agn_schroder2010poly}, the material behaviour of an elastic solid is described by a potential energy function
\[
	W\col\GLp(n)\to\R\,,\quad F\mapsto W(F)
\]
defined on the group $\GLp(n)$ of invertible matrices with positive determinant; here, $F=\grad\varphi$ represents the \emph{deformation gradient} corresponding to a deformation $\varphi\col\Omega\to\R^n$ of an elastic body $\Omega\subset\R^n$, which is assumed to be non-singular and orientation preserving. While this restriction of $W$ to a subset of $\Rnn$ poses additional challenges for the application of variational methods, the above concepts of generalized convexity can still be utilized \cite{ball1976convexity}, partly due to the fact that the set $\GLpn$ is itself rank-one convex and polyconvex.

An elastic energy potential $W\col\GLp(n)\to\R$ is also often assumed to be \emph{objective} (or \emph{frame-indifferent}) as well as \emph{isotropic}, i.e.\ it satisfies
\begin{equation}
	W(Q_1F\.Q_2) = W(F)
	\qquad\text{for all }\;F\in\GLp(n)
	\quad\text{and all }\;\; Q_1,Q_2\in\SO(n)\,,
\end{equation}
where $\SO(n)=\{X\in \Rnn \setvert Q^T Q=\id\,,\;\det Q=1\}$ denotes the special orthogonal group. In terms of sets of matrices, the theory of nonlinear elasticity thereby directly motivates the consideration of the particular class of \emph{isotropic subsets} of $\GLpn$.

\begin{definition}
\label{definition:isotropicSetGLp}
	A set $\Mbb\subset\GLpn$ is called \emph{isotropic} if
	\[
		Q_1F\.Q_2\in \Mbb
		\qquad\text{for all }\;F\in M
		\quad\text{and all }\;\; Q_1,Q_2\in\SO(n)\,.
	\]
\end{definition}

\begin{remark}
\label{remark:isotropyHemitropyOfSets}
	While Definition \ref{definition:isotropicSetGLp} is in agreement with the notion of isotropic subsets of $\GLpn$ employed by Cardaliaguet and Tahraoui \cite{cardaliaguet2002equivalence2}, different notions of isotropy for arbitrary subsets of $\Rnn$ can be found in the literature as well. In particular, there is no clear consensus on whether isotropy of $\Mbb\subset\Rnn$ encompasses only the invariance under multiplication with $Q\in\SOn$ or with any $Q\in\On$, where $\On=\{X\in\Rnn\setvert X^TX=\id\}$ denotes the orthogonal group.\footnote{%
		If the term \emph{isotropy} is used to denote $\On$-invariance, then $\SOn$-invariance is sometimes called \emph{hemitropy} instead.%
	}
	In the following, we will therefore use the terms \enquote{$\SOn$-invariant} and \enquote{$\On$-invariant} if there is any danger of ambiguity.
\end{remark}

Similar to the case of isotropic energy functions (cf.~Section \ref{sectionContains:energySVrepresentation}), any isotropic (in the sense of Definition \ref{definition:isotropicSetGLp}) set $\Mbb\subset\GLpn$ can be expressed in terms of \emph{singular values}. More specifically, there exists a unique set $M\subset\Oset$ such that $F\in\Mbb$ if and only if
\begin{align*}
	\Mbb &= \Mbb(M) \colonequals \{ F\in\GLpn \setvert \svmap(F)\in M \}\,,\\[-0.7em]
	\intertext{where}\\[-2em]
	\Oset &= \{ (\lambda_1,\dotsc,\lambda_n)\in\R^n \setvert \lambda_1\geq\ldots\geq\lambda_n>0 \}
\end{align*}
and $\svmap\col\GLpn\to\Oset$ denotes the mapping of any $F\in\GLp(2)$ to the vector $\svmap(F)$ of its singular values in descending order. Similarly, for any $M\subset\Oset$, the set $\Mbb(M)$ is isotropic.

In the planar case $n=2$, a number of results concerning the generalized convexity properties have been established for isotropic sets of matrices. Cardaliaguet and Tahraoui \cite{cardaliaguet2002equivalence2} characterized the relation between rank-one convexity and polyconvexity of compact isotropic sets.

\begin{proposition}[{Cardaliaguet and Tahraoui \cite{cardaliaguet2002equivalence2}}]
\label{proposition:cardaliaguetMainResult}
	Let $M\subset\Oset[2]$ be compact. Then the set $\mathbb{M}(M)$ is rank-one convex if and only if the two following properties are satisfied:
	\begin{itemize}
		\item[i)] if $C$ is a connected component of $M$, then $\mathbb{M}(C)$ is polyconvex;
		\item[ii)] if $C_1$ and $C_2$ are two distinct connected components of $M$, then $C_1$ and $C_2$ can be strictly separated in the following sense:\footnote{%
			Note that Cardaliaguet and Tahraoui state their result in terms of singular values in \emph{ascending} order.%
		}
		either $\displaystyle\sup_{(x,y)\in C_1} x < \inf_{(x,y)\in C_2} y$\;\; or \;$\displaystyle\sup_{(x,y)\in C_2} x < \inf_{(x,y)\in C_1} y$.\directqed
	\end{itemize}
\end{proposition}

In particular, Proposition \ref{proposition:cardaliaguetMainResult} yields the equivalence of rank-one convexity and polyconvexity of $\Mbb\subset\GLp(2)$ if the set $\Mbb$, and thus its singular value representation $M\subset\Oset[2]$ with $\Mbb=\Mbb(M)$, is compact and connected.

\begin{proposition}[\cite{cardaliaguet2002equivalence2}]
\label{proposition:rankOnePolyEquivalenceCompactConnectedSetsGLp}
	Let $\Mbb\subset\GLp(2)$ be isotropic, compact and connected. Then $\Mbb$
	is rank-one convex if and only if $\Mbb$ polyconvex.
	\directqed
\end{proposition}

By a counterexample \cite[p.~1225]{cardaliaguet2002equivalence2}, Cardaliaguet and Tahraoui also showed that the connectedness required in Proposition \ref{proposition:rankOnePolyEquivalenceCompactConnectedSetsGLp} cannot be simply omitted. In a companion paper \cite{cardaliaguet2002equivalence1}, they also considered the \emph{unconstrained} case of $\OO(2)$-invariant sets of matrices, i.e.\ subsets $\Mbb\subset\R^{2\times2}$ of the form
\begin{equation}\label{eq:unconstrainedSetSingularValueRepresentation}
	\Mbb = \{ X\in\R^{2\times2} \setvert (\lambdahat_1(X),\lambdahat_2(X))\in M \}
	\qquad\text{with }\;
	M\subset \{ (x,y)\in \R^2 \setvert x\geq y\geq 0 \}\,,
\end{equation}
where $\lambdahat_1(X)\geq\lambdahat_2(X)\geq0$ are the singular values of $X\in\R^{2\times2}$ in descending order, obtaining a result analogous to Proposition \ref{proposition:rankOnePolyEquivalenceCompactConnectedSetsGLp}.

\begin{proposition}[\cite{cardaliaguet2002equivalence1}]
\label{proposition:cardliaguetOinvariantSetsUnconstrained}
	Let $\Mbb\subset\R^{2\times 2}$ be $\OO(2)-invariant$, compact, connected and rank-one convex. Then $\Mbb$ is polyconvex.
	\directqed
\end{proposition}

An interesting generalization of Proposition \ref{proposition:cardliaguetOinvariantSetsUnconstrained} has been given by Conti et al.\ \cite{conti2003polyconvexity}.

\begin{proposition}[{Conti et al.\ \cite[Theorem 2.1]{conti2003polyconvexity}}]
\label{proposition:contiSOinvariantSets}
	Let $\Mbb\subset\R^{2\times 2}$ be $\SO(2)$-invariant, compact, connected and rank-one convex. Then $\Mbb$ is polyconvex.
	\directqed
\end{proposition}

Again, note carefully that $\Mbb\subset\R^{2\times2}$ is of the form \eqref{eq:unconstrainedSetSingularValueRepresentation} if and only if it is $\OO(2)$-invariant, whereas the weaker requirement of $\SO(2)$-invariance in Proposition \ref{proposition:contiSOinvariantSets} does not allow for this particular single value representation in general.

More recently, it was also shown by Heinz that rank-one convexity of an $\OO(2)$-invariant compact set $\Mbb\subset\R^{2\times2}$ implies the \emph{quasiconvexity} \cite{heinz2015quasiconvexity} of $\Mbb$; note that Heinz' result does not require $\Mbb$ to be connected.

\bigskip

In terms of energy functions on $\GLp(2)$, Proposition \ref{proposition:rankOnePolyEquivalenceCompactConnectedSetsGLp} yields an immediate consequence for the generalized notion of \emph{q-convexity}.\footnote{%
	In the literature, functions satisfying Definition \ref{definition:q-Convexity} are more commonly known as \emph{quasiconvex} functions; however, due to the obvious ambiguity accompanied by the term, we will exclusively refer to this property as q-convexity here and throughout.%
}

\begin{definition}\label{definition:q-Convexity}
	Let $W\col M\to\R_\infty$ be an extended real-valued function on a convex subset $M$ of a linear space. Then $W$ is called \emph{q-convex} if for every $c\in\R$, the sublevel set $S_c \colonequals \{x\in M \setvert W(x) \leq c\}$ is convex.
\end{definition}
While every convex function is q-convex, the reverse implication does not hold in general; for example, the mapping $t\mapsto \ln^2(t)$ is q-convex on $\Rp=(0,\infty)$ but not convex.

The concept of q-convexity can be generalized to weakened convexity conditions, including rank-one convexity and polyconvexity.

\begin{definition}\label{definition:q-ConvexityGeneralized}
	A function $W\col \Rnn\to\R_\infty$ is called \emph{q-rank-one-convex} [\emph{q-polyconvex}] if all sublevel sets of $W$ are rank-one convex [polyconvex].
\end{definition}

Since the sublevel sets of any isotropic function $W\col\GLp(2)\to\R$ are isotropic, the following corollary follows immediately from Proposition \ref{proposition:rankOnePolyEquivalenceCompactConnectedSetsGLp}; note the obvious analogy to Conjecture \ref{conjecture:mielke}.

\begin{corollary}
\label{corollary:qConvexityEquivalence}
	Let $W\col\GLp(2)\to\R_\infty$ be an isotropic function such that all sublevel sets of $W$ are compact and connected. Then $W$ is q-rank-one-convex if and only if $W$ is q-polyconvex.
	\directqed
\end{corollary}
%
%
%
%
\subsection{Polyconvexity criteria in terms of singular values}

In the aforementioned 2005 article \cite{mielke2005necessary}, Mielke obtained a number of results for generalized convexity properties of isotropic energy functions on $\GLpn$ in terms of their singular value representation, including a necessary and sufficient criterion for polyconvexity in the planar case suited for both numerical and analytical applications (cf.\ Section \ref{section:assemblingTheProof}).

Mielke considered the representation of an isotropic energy $W\col\GLp(2)\to\R$ in terms of \emph{ordered} singular values, i.e.\ the uniquely determined function
\label{sectionContains:energySVrepresentation}
\begin{equation}
	\ghat\col\Oset[2]\to\R
	\qquad\text{on the set }\;\;
	\Oset[2]
	= \{ (\lambda_1,\lambda_2)\in\R^2 \setvert \lambda_1\geq\lambda_2>0 \}
	\subset \Rp^2
\end{equation}
such that
\begin{equation}\label{eq:mielkeOrderedSingularValueRepresentation}
	W(F) = \ghat(\svmap(F))
	\qquad\tforall F\in\GLp(2)\,,
\end{equation}
where $\Rp=(0,\infty)$ and $\svmap\col\GLp(2)\to\Oset[2]$ is the mapping of any $F\in\GLp(2)$ to the vector $\svmap(F)$ of its singular values in descending order. After establishing general necessary and sufficient criteria for the polyconvexity\footnote{%
	It is important to note that the definition of polyconvexity of $W\col\GLpn\to\R$ employed by Mielke requires $W$ to satisfy the growth condition $W(F)\to+\infty$ for $\det F\to0$, which corresponds to the lower semicontinuity of the extension of $W$ to $\Rnn$ obtained by setting $W(F)=+\infty$ for all $F\in\Rnn\setminus\GLpn$, cf.\ Remark \ref{remark:compactnessGrowthCondition}.%
}
of functions, Mielke \cite[Theorem 4.1]{mielke2005necessary} also obtained a previous result by Miroslav \v{S}ilhav\'{y} for the differentiable planar case as a direct corollary.
\begin{theorem}[{\v{S}ilhav\'{y} \cite[Proposition 4.1]{vsilhavy2002convexity}}]\label{polycrit}
	Let $\ghat\in C^1(\Oset[2];\R)$. Then the function $W\col\GLp(2)\to\R$ represented by $\ghat$ in terms of ordered singular values as defined by \eqref{eq:mielkeOrderedSingularValueRepresentation} is polyconvex if and only if
	\begin{samepage}
		\begin{equation}
			\forall\;\gamma\in \Oset[2]
			\quad\exists\;c\in \left[
				-\frac{\frac{\partial \ghat}{\partial{\lambda_1}}(\gamma_1,\gamma_2)-\frac{\partial \ghat}{\partial{\lambda_2}}(\gamma_1,\gamma_2)}{\gamma_1-\gamma_2}
				\,,\;
				\frac{\frac{\partial \ghat}{\partial\lambda_1}(\gamma_1,\gamma_2)+\frac{\partial \ghat}{\partial\lambda_2}\ghat(\gamma_1,\gamma_2)}{\gamma_1+\gamma_2}
			\right]
			\quad\forall\;\nu\in \Oset[2]:
		\end{equation}
		\[
			\quad\ghat(\nu_1,\nu_2)
			\geq \ghat(\gamma_1,\gamma_2) + \frac{\partial\ghat}{\partial{\lambda_1}}(\gamma_1,\gamma_2)\cdot(\nu_1-\gamma_1) + \frac{\partial\ghat}{\partial{\lambda_2}}(\gamma_1,\gamma_2)\cdot(\nu_2-\gamma_2) + c\,(\nu_1-\gamma_1)\,(\nu_2-\gamma_2)
			\,.
			\directqed
		\]
	\end{samepage}
\end{theorem}
%
%
%
%
%
\subsection{Previous examples of rank-one convex, non-polyconvex energy functions}
\label{section:previousExamples}

A number of examples for rank-one convex, non-polyconvex functions have already been given in the literature \cite{aubert1987counterexample,alibert1992example,dacorogna1988counterexample}. The first such counterexample, a fourth-degree homogeneous polynomial expression, is due to Aubert \cite{aubert1986contribution,aubert1987counterexample}, who considered the function $\WA$ given by
\begin{align}\label{eq:aubertEnergy}
	\WA\col\GLp(2)\to\R\,,\quad
	\WA(F)
	&=\frac13\.\norm{F}^4-\frac16\.(\det F)^2-\frac23\.\det F\cdot\norm{F}^2
	\\&= \frac13(\lambda_1^4+\lambda_2^4) + \frac12\lambda_1^2\lambda_2^2 - \frac23(\lambda_1^3\lambda_2+\lambda_1\lambda_2^3)\notag
\end{align}
for all $F\in\GLp(2)$ with (not necessarily ordered) singular values $\lambda_1,\lambda_2$. 

\begin{figure}[h!]\begin{center}
		\begin{minipage}[t]{0.45\linewidth}
			\centering
			\includegraphics[height=.8\textwidth]{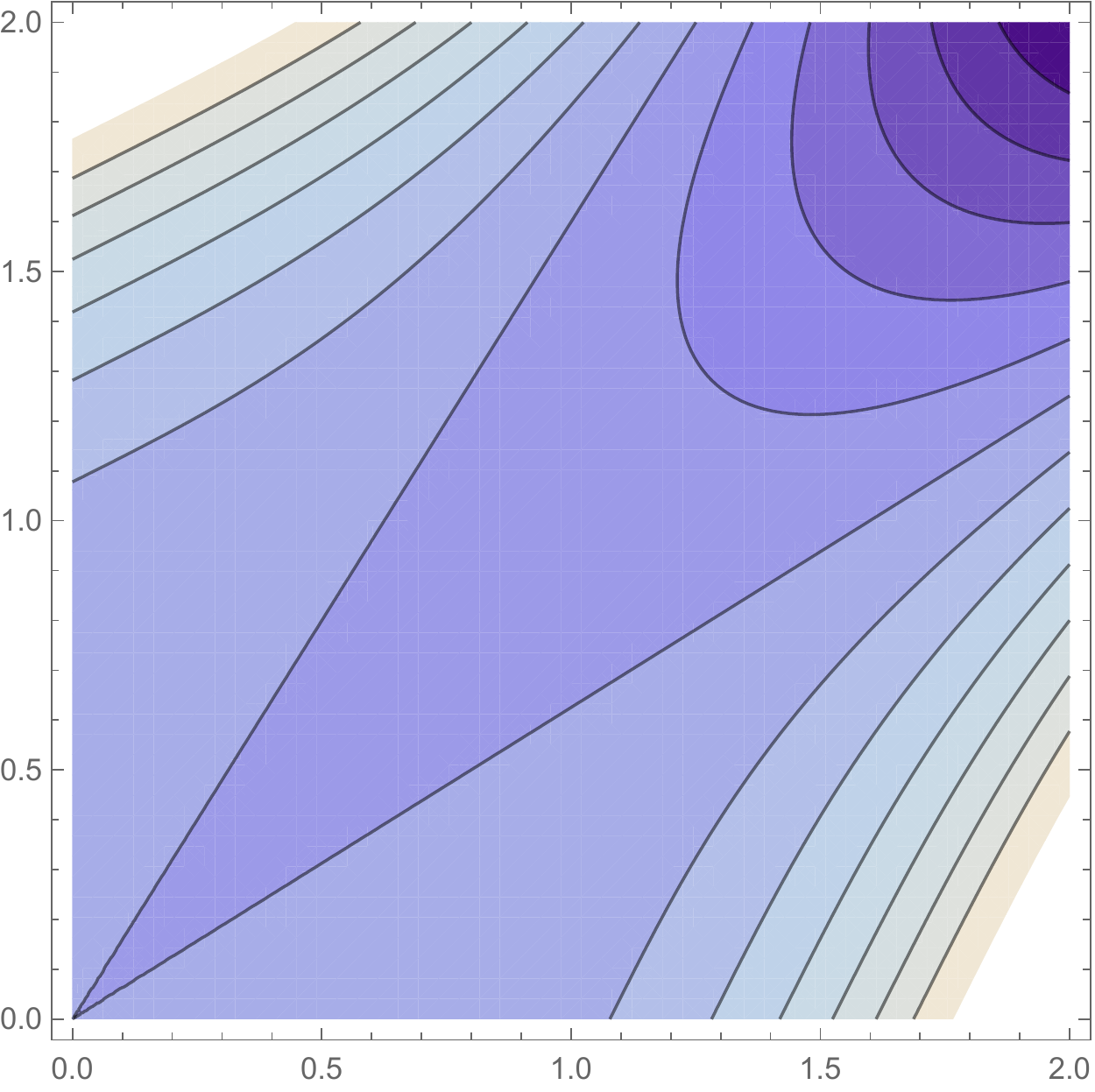}
			\caption{The sublevels of the energy $\WA$ on $\GLp(2)$ are connected but not compact.}
			\label{Fig3}
		\end{minipage}
		\hfill
		\begin{minipage}[t]{0.450\linewidth}
			\centering
			\includegraphics[height=.8\textwidth]{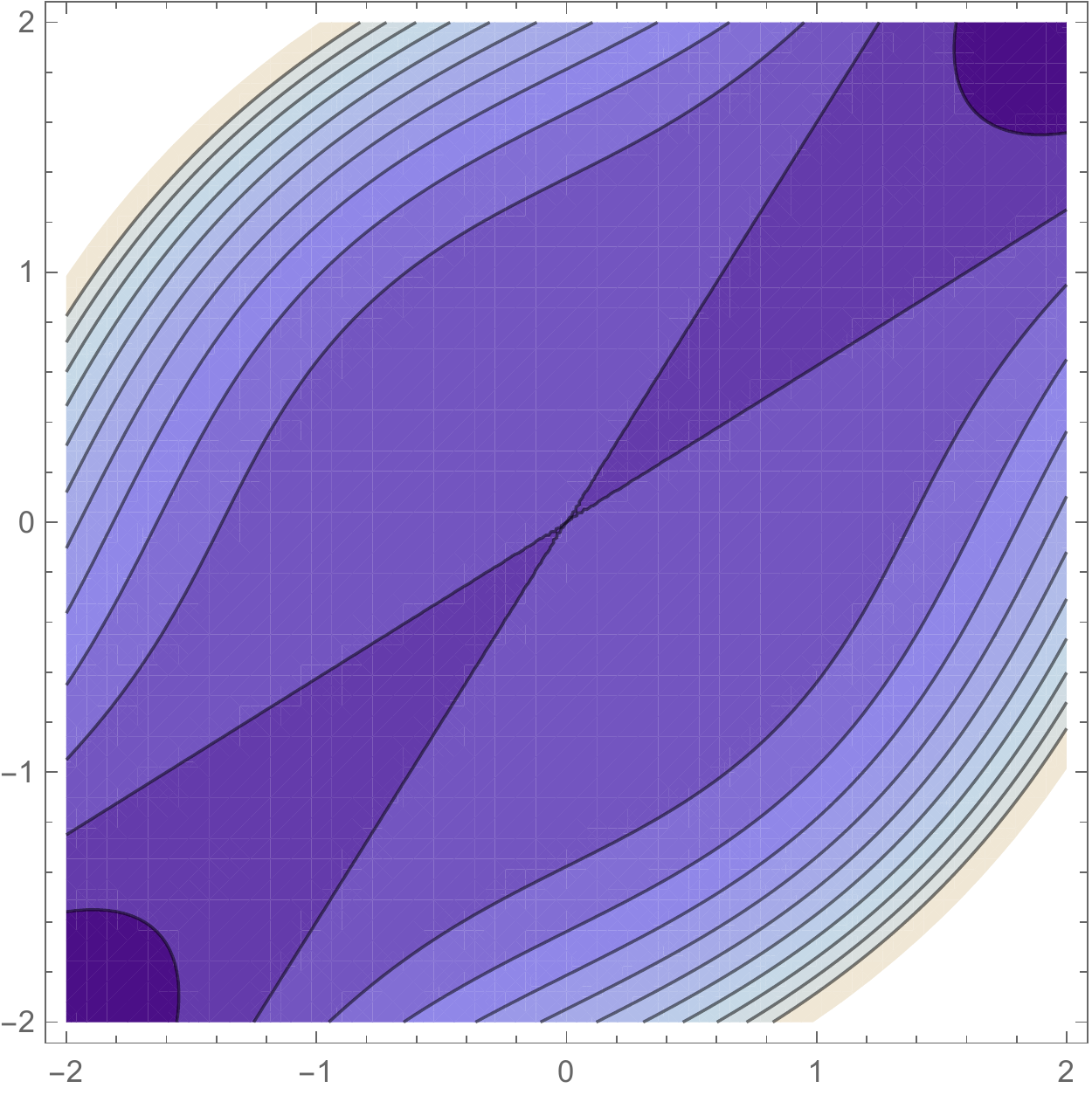}
			\caption{The sublevels of the extension of $\WA$ to $\R^{2\times 2}$ are neither connected nor compact.}
			\label{Fig4}
			\label{fig:firstDiagonalFigure}
		\end{minipage}	
	\end{center}
\end{figure}

However, the function \eqref{eq:aubertEnergy} is not a counterexample to Conjecture \ref{conjecture:mielke}, since the sublevel sets are not compact;\footnote{In Appendix \ref{appendix:aubertConnectedness}, we show that the sublevel sets are indeed connected.} simply note that
\begin{equation}\label{eq:nonCompactExample}
	\WA(\diag(t,t)) = -\frac{t^4}{6}\,,
\end{equation}
i.e.\ that each sublevel set contains an (unbounded) set of the form $\{a\cdot\id\setvert a>d\}$ for some $d\geq0$, as visualized\footnote{%
	Figs.~\ref{Fig3}, \ref{Fig4} and \ref{fig:firstDiagonalFigureInBlock}--\ref{fig:lastDiagonalFigure} show the graphs and sublevel set contours for different energies applied to \emph{diagonal} matrices, i.e.\ visualize the mappings $(x,y)\mapsto W(\diag(x,y))$ for some isotropic energy function $W$, with $(x,y)\in\R^2$ if $W$ is defined on $\R^{2\times2}$ or $(x,y)\in\Rp^2$ if $W$ is defined on the domain $\GLp(2)$. In the latter case, the figures can equivalently be interpreted as showing $W$ in terms of (unordered) singular values. Here and throughout, lower sublevel sets are represented by darker regions.%
}
in Figs.~\ref{Fig3} and \ref{Fig4}. Indeed, such negative growth conditions are often employed to simplify the otherwise cumbersome task of showing that a function is not polyconvex \cite{ball1976convexity,Dacorogna08,alibert1992example} (cf.~\cite[Remark (1)]{aubert1987counterexample}). Therefore, the examples of rank-one convex functions which are not polyconvex encountered in the literature are usually not suitable to falsify Conjecture \ref{conjecture:mielke}.

This observation is highlighted further by another important example (which is a homogeneous polynomial expression of grade 4 as well): for the family of isotropic energies
\begin{equation}\label{eq:dacorognaEnergy}
	\WADM^\gamma\col\R^{2\times2}\to\R\,,\qquad \WADM^\gamma(F) = \norm{F}^2\.(\norm{F}^2-2\gamma\.\det F)\,,\qquad \gamma\in\R\,,
\end{equation}
it was shown by Alibert, Dacorogna and Marcellini \cite{alibert1992example,dacorogna1988counterexample} that
\begin{align*}
	\WADM^\gamma \quad \text{is convex on } \R^{2\times2}\quad &\iff \quad \abs{\gamma}\leq \frac{2\sqrt{2}}{3}\approx 0.942809\,,\\[-0.3em]
	\WADM^\gamma\quad \text{is polyconvex on } \R^{2\times2}\quad &\iff \quad\abs{\gamma}\leq 1\,,\vphantom{\frac23}\\[-0.3em]
	\WADM^\gamma \quad \text{is quasiconvex on } \R^{2\times2}\quad &\iff \quad\abs{\gamma}\leq \gamma_q \ \ \text{with}\ \ \gamma_q>1\,,\vphantom{\frac23}\\[-0.3em]
	\WADM^\gamma \quad \text{is rank-one convex on } \R^{2\times2}\quad &\iff \quad \abs{\gamma}\leq \frac{2}{\sqrt{3}}\approx 1.1547
	\,.
\end{align*}
Again, in the non-polyconvex case $\abs{\gamma}>1$, the sublevel sets are not compact, since $\WADM(a\cdot\id)=4\.a^4(1-\gamma)<0$ for all $a>0$ if $\gamma>1$ and $\WADM(\diag(a,-a))=4\.a^4(1+\gamma)<0$ for all $a>0$ if $\gamma<-1$, where $\diag(x,y)\in\R^{2\times2}$ denotes the diagonal matrix with diagonal entries $x,y\in\R$.

\tikzset{
	state/.style={
	minimum height = 0.55cm, 
	outer sep = 0,
	inner sep = 0,
	black
	}
}
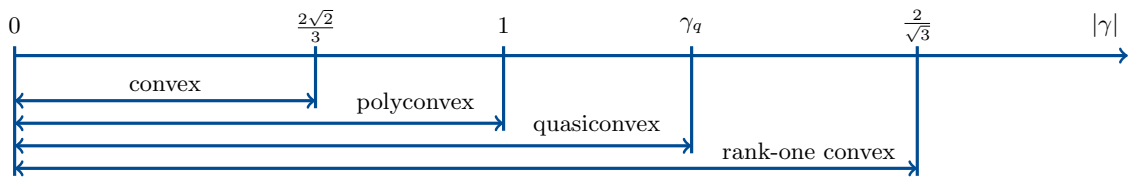
\begin{figure}[h!]
\centering
\small
\begin{tikzpicture}
	\node[state](1){$0$};
	\node[right of=1, node distance = 4cm, state](2){$\frac{2\sqrt{2}}{3}$};
	\node[right of=1, node distance = 6.5cm, state](3){$1$};
	\node[right of=1, node distance = 9cm, state](4){$\gamma_q$};
	\node[right of=1, node distance = 12cm, state](5){$\frac{2}{\sqrt{3}}$};
	\node[right of=1, node distance = 14.5cm, state](6){$\abs\gamma$};
	\path[very thick, udedarkblue]
		(1) edge ++(0,-2cm)
		(2) edge ++(0,-1.1cm)
		(3) edge ++(0,-1.4cm)
		(4) edge ++(0,-1.7cm)
		(5) edge ++(0,-2cm)
		(1)++(0,-0.4cm) edge[->] ($(6)+(0.3cm,-0.4cm)$)
		(1)++(0,-1cm) edge[<->] node[black,pos=0.5,above]{convex} ($(2)-(0,1cm)$)
		(1)++(0,-1.3cm) edge[<->] node[black,pos=0.82,above]{polyconvex} ($(3)-(0,1.3cm)$)
		(1)++(0,-1.6cm) edge[<->] node[black,pos=0.86,above]{quasiconvex} ($(4)-(0,1.6cm)$)
		(1)++(0,-1.9cm) edge[<->] node[black,pos=0.88,above]{rank-one convex} ($(5)-(0,1.9cm)$)
	;
\end{tikzpicture}
\caption{A schematic overview for the Alibert-Dacorogna-Marcellini family of energies. It is currently unknown whether $\gamma_q<\frac{2}{\sqrt{3}}$.}\label{schita}
\end{figure}

Now, for $\gamma\in\R$, consider the restriction 
\begin{equation}\label{eq:dacorognaEnergy2}
	\WADMp^\gamma = \WADM^\gamma\big|_{\GLp(2)}\col\GLp(2)\to\R\,,\quad \WADM^\gamma(F) = \norm{F}^2\.(\norm{F}^2-2\gamma\.\det F)
\end{equation}
of $\WADM^\gamma$ to $\GLp(2)$. Then $\WADMp^\gamma$ can be expressed in terms of (unordered) singular values $\lambda_1,\lambda_2$ of $F\in\GLp(2)$ via
\begin{equation}\label{eq:dacorognaEnergy2singularValues}
	\WADM^\gamma(F)=(\lambda_1^2+\lambda_2^2)^2-2 \,\gamma\,(\lambda_1^2+\lambda_2^2) \,\lambda_1 \lambda_2\,.
\end{equation}
In this case, the sublevel sets are not compact for any $\gamma\in\R$; observe that for every $c>0$, since
\[
	\lim_{n\to\infty}\,\WADMp\Bigl(\frac1n\cdot\id\Bigr)
	= \lim_{n\to\infty}\,\frac{4\.(1-\gamma)}{n^4}
	= 0\,,
\]
we find $\frac1n\cdot\id\in S_c$ for all sufficiently large $n\in\N$, thus the sublevel set $S_c$ contains a sequence without a subsequence which converges in $\GLp(2)$.

\begin{remark}
\label{remark:compactnessGrowthCondition}
	From generalizing the above argument, it follows that the compactness of sublevel sets for an energy on $\GLp(2)$ requires $W$ to satisfy the growth condition $W(F)\to\infty$ for $\det F\to0$. In particular, this condition is never satisfied by any polynomial function $W$, which immediately excludes the class of polynomial energies as possible counterexamples to Conjecture \ref{conjecture:mielke}.
\end{remark}

\begin{figure}[h!]\centering
	\begin{minipage}[t]{0.45\linewidth}
		\centering
		\includegraphics[height=.8\textwidth]{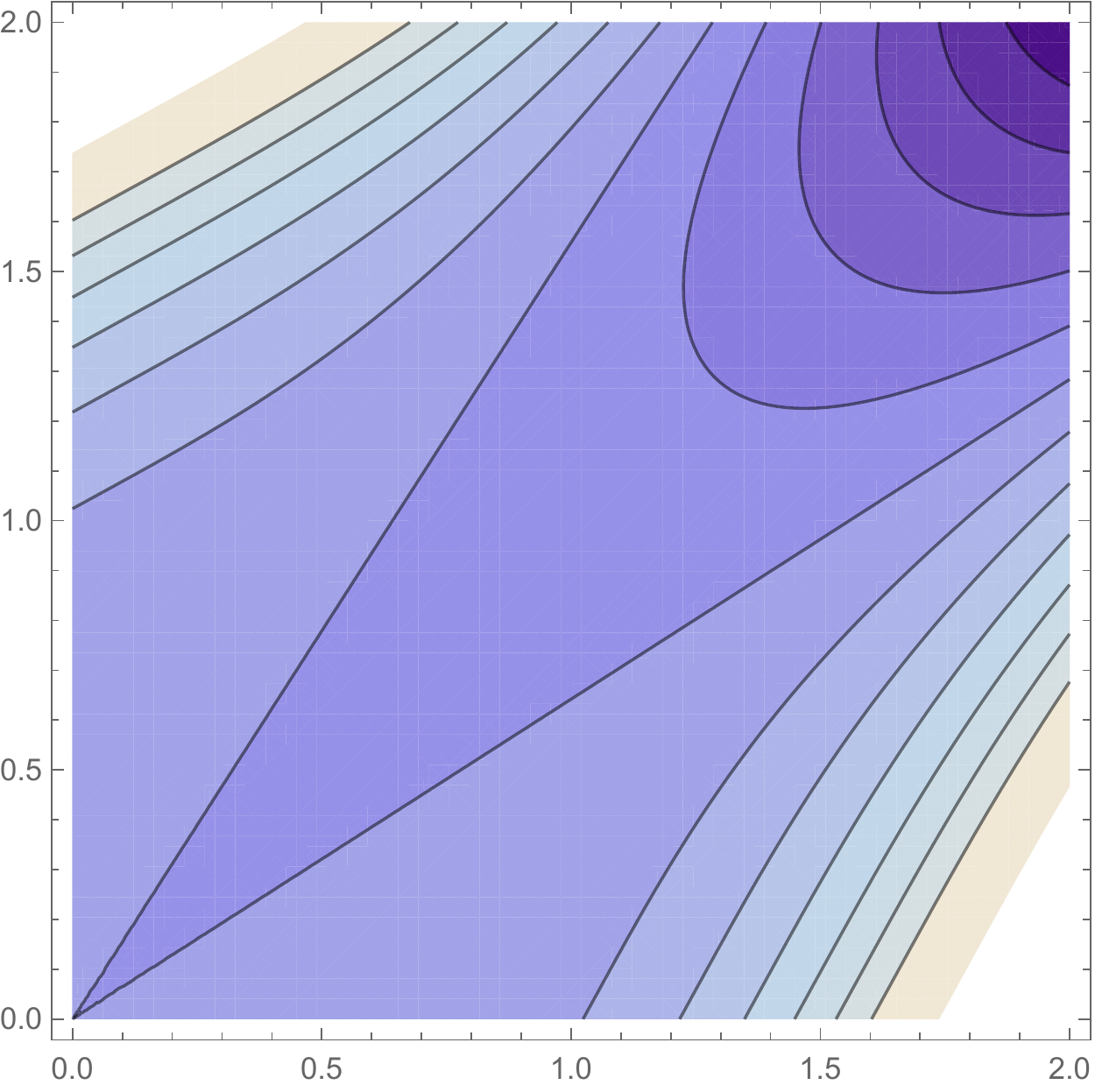}
		\caption{The sublevels of $\WADMp^{1.1}$ are connected, but unbounded and thus not compact.}
		\label{DM2}
		\label{fig:firstDiagonalFigureInBlock}
	\end{minipage}
	\hfill
	\begin{minipage}[t]{0.45\linewidth}
		\centering
		\includegraphics[height=.8\textwidth]{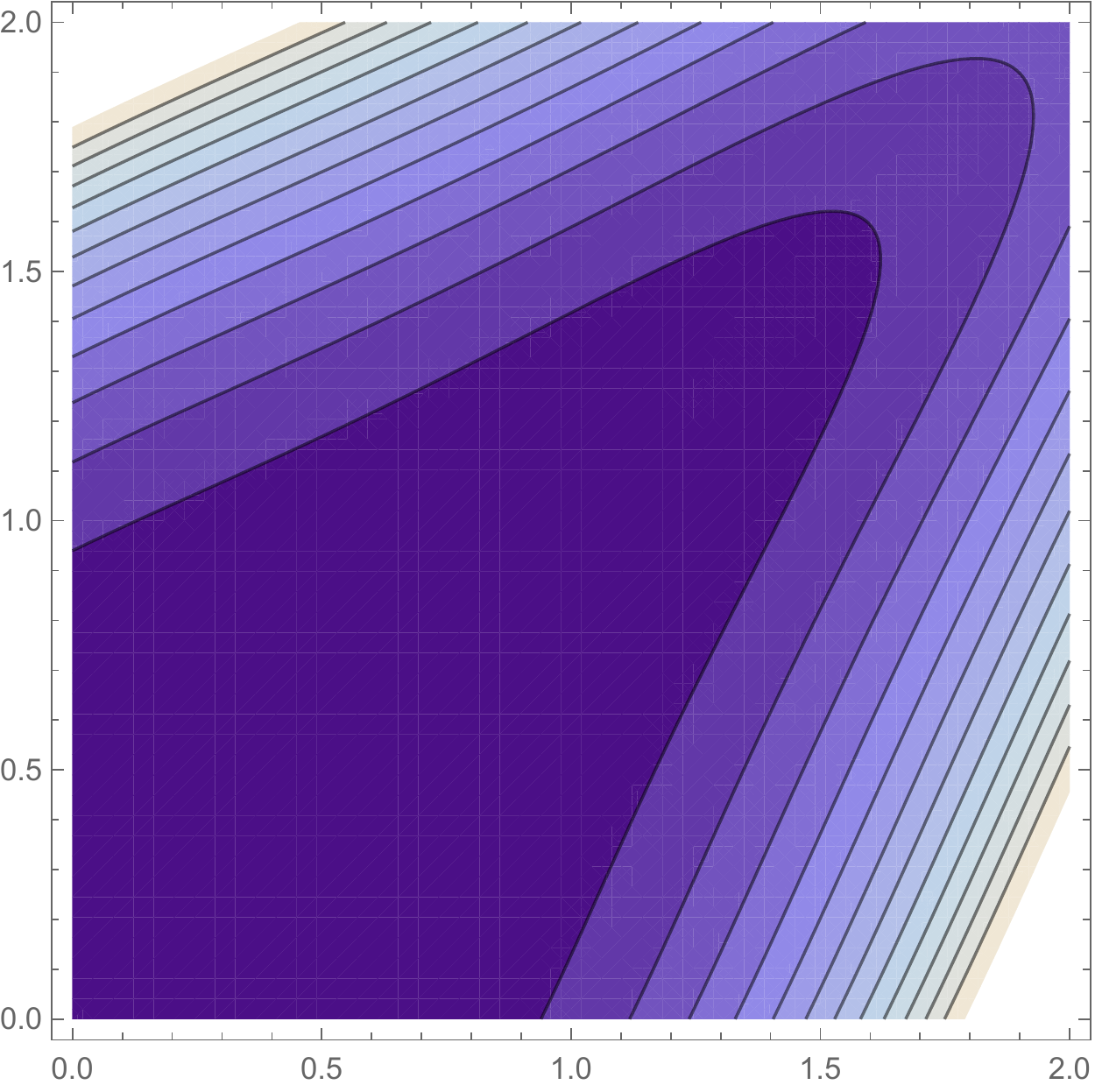}
		\caption{The sublevels of $\WADMp^{0.97}$, which are connected and bounded but not compact.}
		\label{DM1}
	\end{minipage}
\end{figure}

Another example of an $\OO(2)$-invariant energy which is rank-one convex but not polyconvex was introduced by \v{S}ilhav\'{y} \cite{silhavy2002invariant}, who proposed the function $\WS\col\R^{2\times 2}\to \R$ given by
\[
	\WS\col\R^{2\times 2}\to \R\,,\quad
	\WS(F)
	=\ghatS(\svmap(F))
	\qquad\text{with}\qquad
	\ghatS\col\Oset[2]\to\R\,,\quad
	\ghatS(\lambdahat_1,\lambdahat_2)
	=\begin{cases}
		\lambdahat_1\lambdahat_2 &:\lambdahat_1\leq 1\\
		\lambdahat_1+\lambdahat_2 -1&: \lambdahat_1\geq 1\\
	\end{cases}
\]
in terms of ordered singular values $\lambdahat_1,\lambdahat_2$. Again, if we consider the restriction $\WSp=\WS\restrict{\GLp(2)}$ of $\WS$ to $\GLp(2)$, then the sublevel sets (cf.~Figs.~\ref{FigS3} and \ref{FigS4}) are generally not compact due to the boundedness of $\WSp(F)$ for $\det F\to0$.

\begin{figure}[h!]\begin{center}
		\begin{minipage}[t]{0.54\linewidth}
			\centering
			\includegraphics[width=1\textwidth]{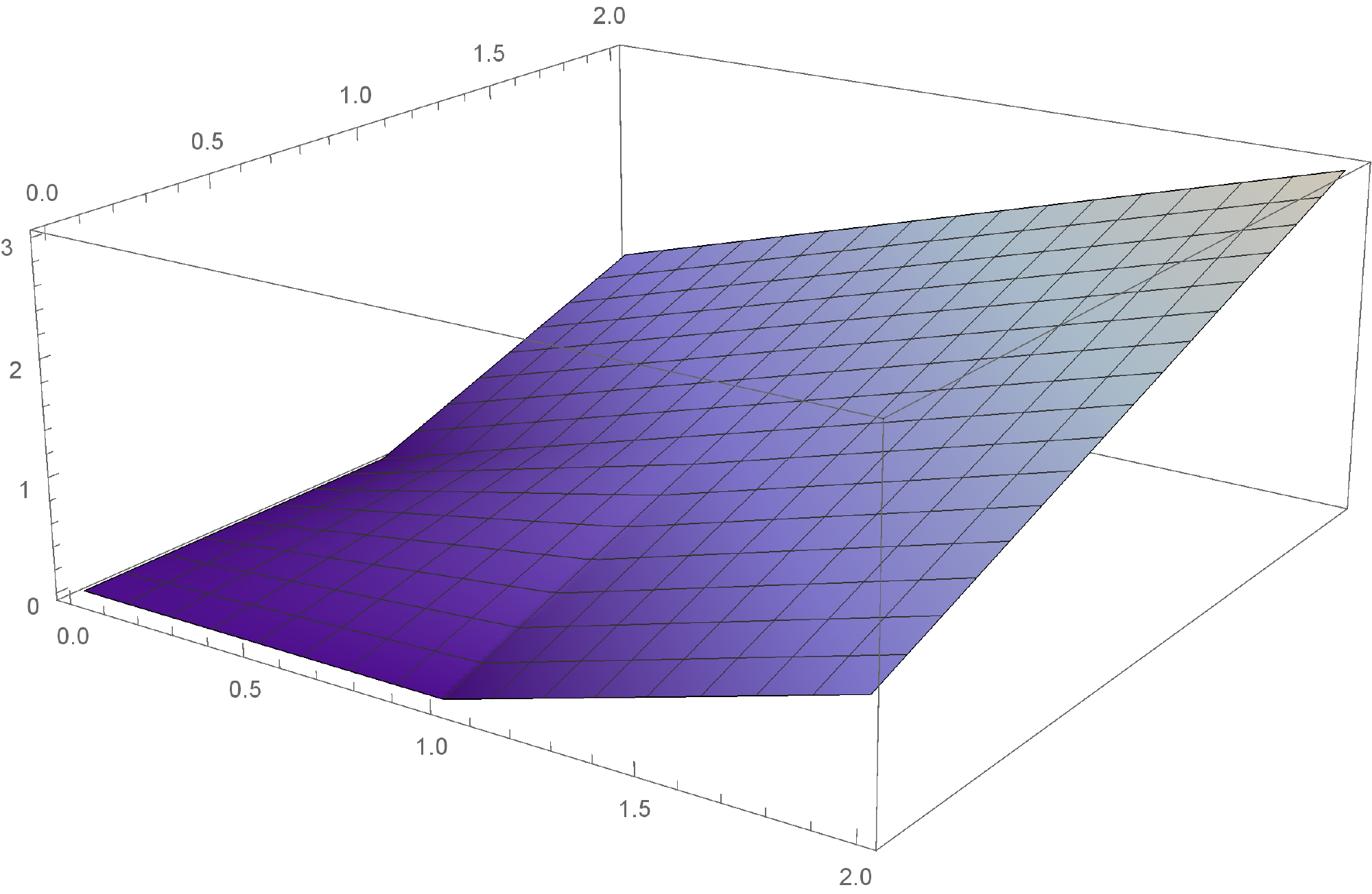}
			\caption{The energy $\WSp(F)$ on $\GLp(2)$ in terms of singular values of $F$.}
			\label{FigS3}
		\end{minipage}
		\hfill
		\begin{minipage}[t]{0.41\linewidth}
			\centering
			\includegraphics[height=0.85\textwidth]{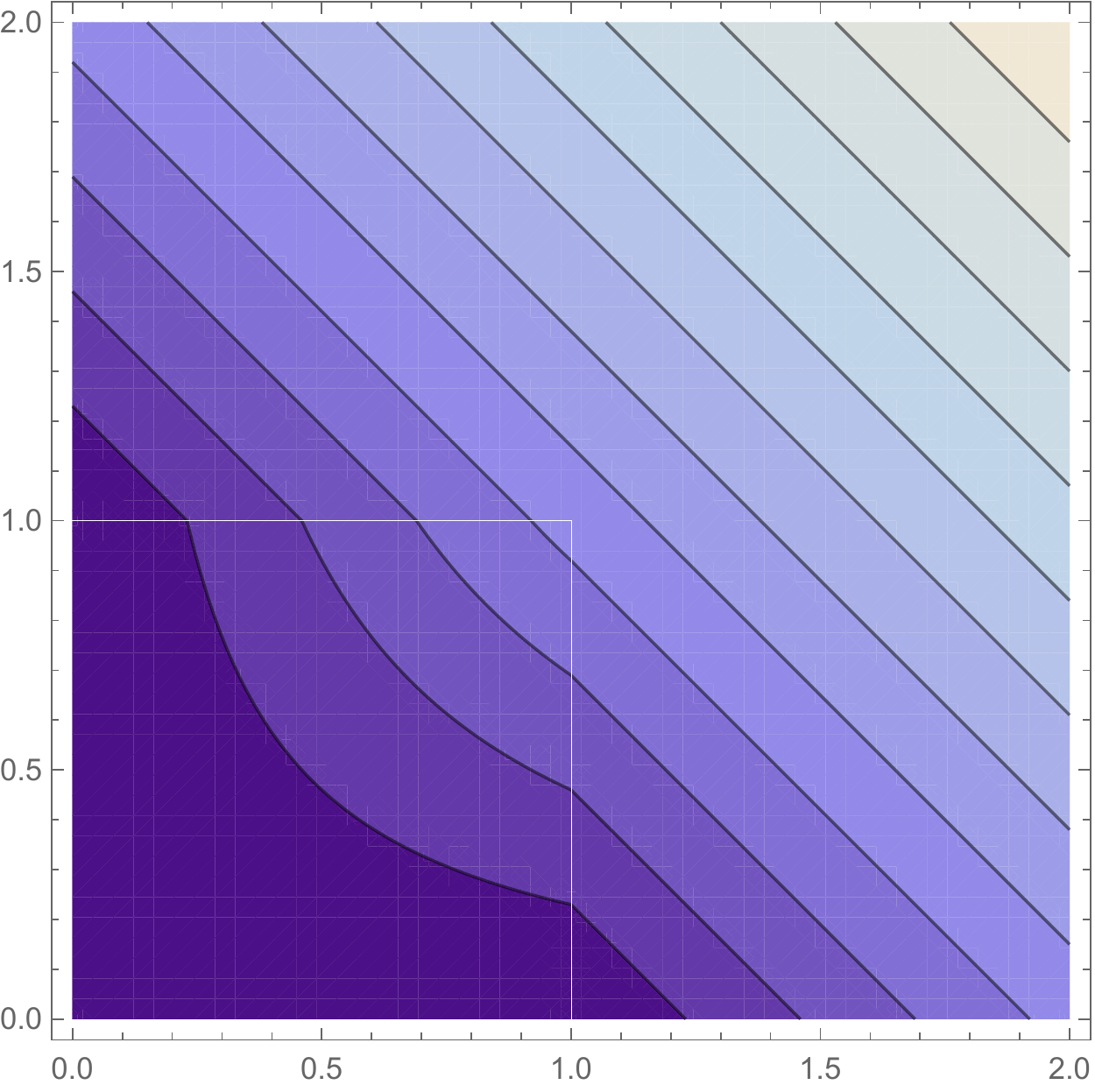}
			\caption{The sublevels of the energy $\WSp$ on $\GLp(2)$ are connected, but not compact.}
			\label{FigS4}
		\end{minipage}	
	\end{center}
\end{figure}
%
%
%
%
%

\pagebreak
\section{The counterexample}\label{section:counterExample}

In the following, we will show that the energy function
\begin{align}
\label{eq:counterexampleDefinition}
	\Wex\col\GLp(2)\to\R\,,\qquad \Wex(F) &= \frac{\opnorm{F}^2}{\det F} - \log\biggl(\frac{\opnorm{F}^2}{\det F}\biggr) + \log(\det F) + \frac{1}{\det F}\\
	&= \frac{\lambdamax}{\lambdamin} - \log\left(\frac{\lambdamax}{\lambdamin}\right) + \log(\lambdamax\lambdamin) + \frac{1}{\lambdamax\lambdamin}\,\nonumber
\end{align}
is not polyconvex, although it satisfies all the conditions posed in Conjecture \ref{conjecture:mielke}, i.e.\ $\Wex$ is isotropic and rank-one convex with connected, compact sublevel sets; here and in the following, we denote the singular values of $X\in\R^{2\times2}$ by $\lambdamax(X)\geq\lambdamin(X)$ and simply write $\lambdamax,\lambdamin$ for the ordered singular values of the energy's argument $F$, while $\opnorm{X}=\sup_{\norm{h}=1}\norm{X\, h}=\lambdamax(X)$ denotes the operator norm of any $X\in\R^{2\times2}$.

In order to establish that the desired properties hold for $\Wex$, we will require a number of auxiliary results concerning the sublevel sets of isotropic functions with a so-called \emph{volumetric-isochoric split}, i.e.\ functions of the form
\begin{equation}\label{eq:volIsoSplitGeneralForm}
	W\col\GLp(2)\to\R\,,\qquad W(F) = \Wiso(F) + \Wvol(\det F)
\end{equation}
with $\Wiso\col\GLp(2)\to\R$ and $\Wvol\col\Rp=(0,\infty)\to\R$ such that $\Wiso$ is \emph{isochoric}, i.e.\ satisfies $W(aF)=W(F)$ for all $a>0$. If $\Wiso$, and therefore $W$, is objective and isotropic, then any function $W$ of the form \eqref{eq:volIsoSplitGeneralForm} can be expressed as \cite{agn_martin2015rank}
\begin{equation}\label{eq:volIsoSplitGeneralFormSingularValues}
	W(F) = \hhat(K(F)) + f(\det F) = \underbrace{\hhat\left(\frac{\lambdamax}{\lambdamin}\right)}_{=\Wiso(F)} + \underbrace{f(\lambdamax\lambdamin) \vphantom{\left(\frac{\lambdamax}{\lambdamin}\right)}}_{=\Wvol(\det F)}
\end{equation}
for all $F\in\GLp(2)$ with singular values $\lambdamax\geq\lambdamin>0$ and with uniquely defined functions $\hhat\col[1,\infty)\to\R$ and $f\col(0,\infty)\to\R$. Here,
\begin{equation}\label{eq:Kdefinition}
	K\col\GLp(2)\to\R\,,\qquad K(F) \colonequals \frac{\opnorm{F}^2}{\det F} = \frac{\lambdamax}{\lambdamin}
\end{equation}
is the \emph{linear distortion} (or \emph{dilation}) function, which plays an important role in the theory of conformal and quasiconformal mappings \cite{astala2008elliptic} as well as for the characterization of rank-one convexity, quasiconvexity and polyconvexity of isochoric planar energy functions \cite{agn_martin2015rank,agn_martin2019envelope}. More specifically, the isochoric part $\Wiso$ of an energy $W$ of the form \eqref{eq:volIsoSplitGeneralForm} is rank-one convex, quasiconvex and polyconvex if and only if the function $\hhat\col[1,\infty)\to\R$ given by \eqref{eq:volIsoSplitGeneralFormSingularValues} is nondecreasing and convex \cite[Theorem 3.3]{agn_martin2015rank}. This equivalence between polyconvexity and rank-one convexity does not hold for the full energy $W$ itself, as the example $\Wex$ will show.

\begin{figure}[h!]\begin{center}
		\begin{minipage}[t]{0.54\linewidth}
			\centering
			\includegraphics[width=1\textwidth]{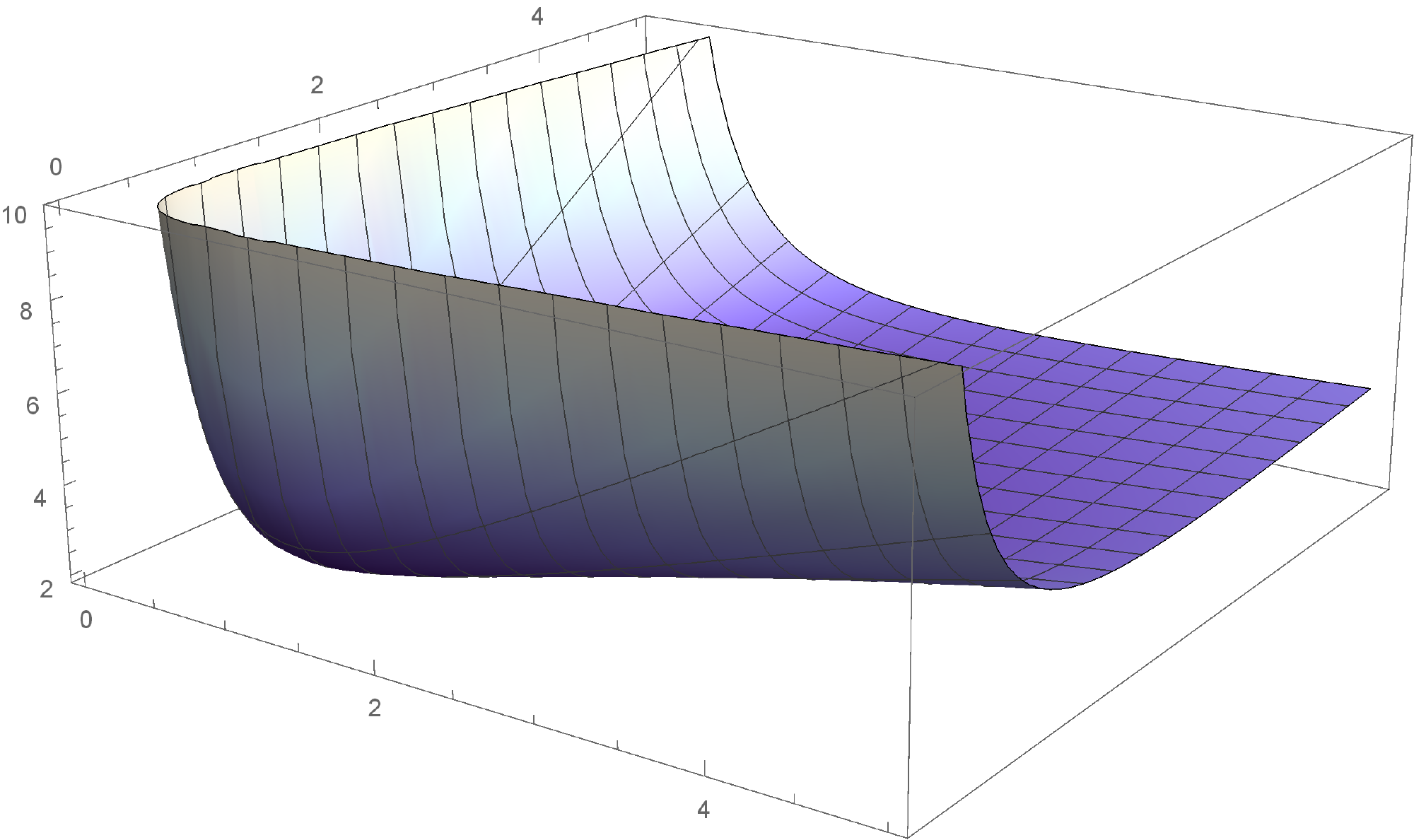}
			\caption{The energy function $\Wex$ in terms of singular values.}
			\label{FigW1}
		\end{minipage}
		\hfill
		\begin{minipage}[t]{0.35\linewidth}
			\centering
			\includegraphics[height=1\textwidth]{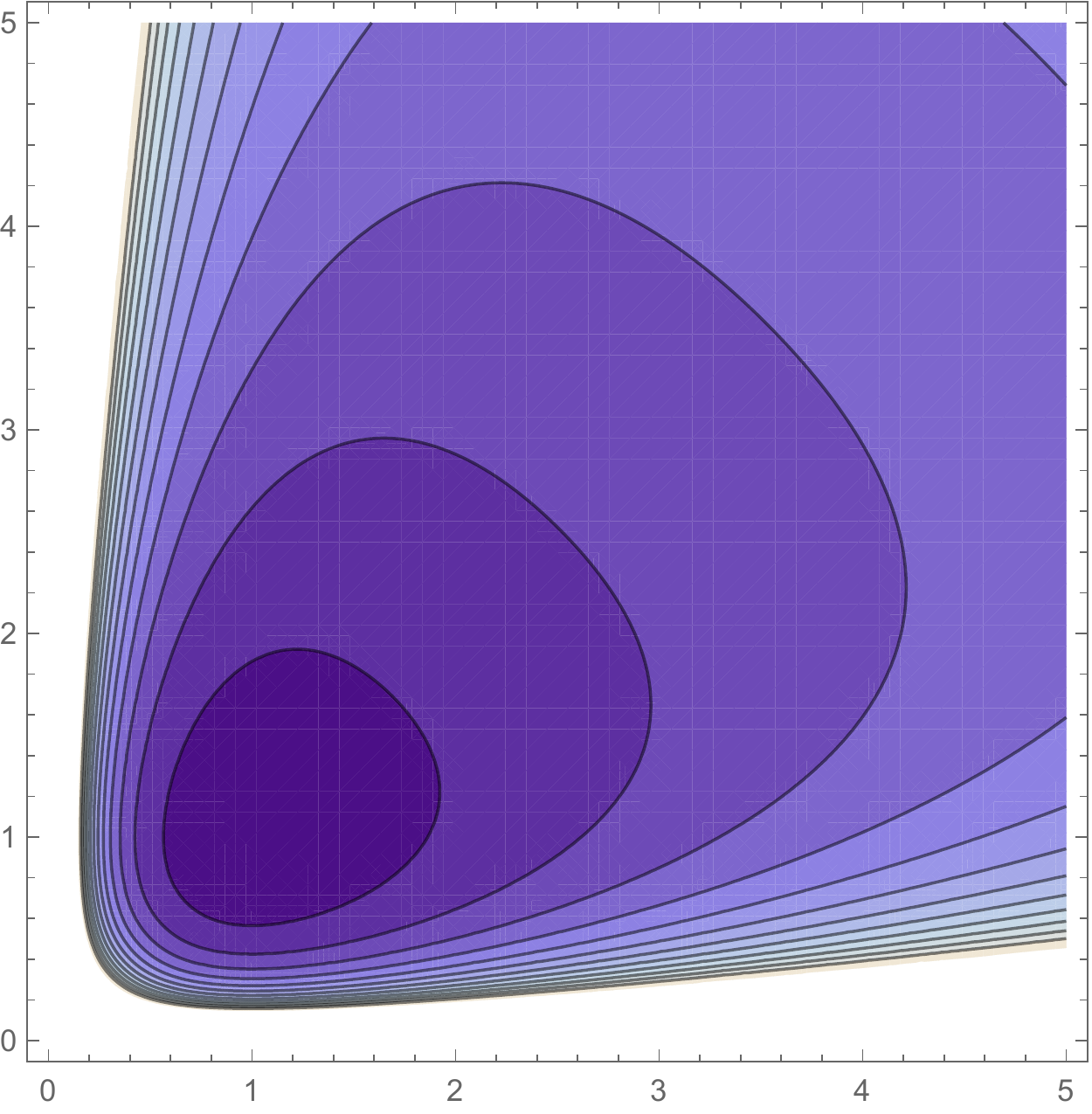}
			\caption{The sublevels of the energy $\Wex$ are compact and connected.}
			\label{FigW2}
			\label{fig:lastDiagonalFigure}
		\end{minipage}	
	\end{center}
\end{figure}

It is clear that $\Wex$ from \eqref{eq:counterexampleDefinition} can be expressed in the form \eqref{eq:volIsoSplitGeneralForm} with
\begin{equation}\label{eq:counterexampleIsoPart}
	 \Wiso(F) = K(F) - \log(K(F)) = \frac{\lambdamax}{\lambdamin} - \log\left(\frac{\lambdamax}{\lambdamin}\right)
\end{equation}
and
\begin{equation}\label{eq:counterexampleVolPart}
	 \Wvol(\det F) = \log(\det F) + \frac{1}{\det F} = \log(\lambdamax\lambdamin) + \frac{1}{\lambdamax\lambdamin}
\end{equation}
or, equivalently, in the form \eqref{eq:volIsoSplitGeneralFormSingularValues} with
\begin{equation}\label{eq:counterexampleVolIsoPartSingularValues}
	\hhat(t) = t-\log(t) \qquad\text{and}\qquad f(t)=\frac1t-\log\left(\frac{1}{t}\right)\,.
\end{equation}

As shown in Figs.\ \ref{FigW1} and \ref{FigW2}, the growth behaviour of $\Wex(F)$ for $\det F\to0$ clearly distinguishes $\Wex$ from the examples considered in Section \ref{section:previousExamples}.

\subsection{Auxiliary results for volumetric-isochorically split energy functions}

In order to establish $\Wex$ as a counterexample to Conjecture \ref{conjecture:mielke}, we first state a sufficient criterion for the compactness of the sublevel sets for energy functions with a volumetric-isochoric split.

\begin{lemma}\label{lemma:compactness}
	Let $W\col\GLp(2)\to\R$ be of the form \eqref{eq:volIsoSplitGeneralFormSingularValues} such that $\hhat,f$ are lower semicontinuous and
	\begin{equation}\label{eq:compactnessGrowthConditions}
		\lim_{t\to\infty}\hhat(t)=\lim_{t\to\infty}f(t)=\lim_{t\to0}f(t)=\infty\,.
	\end{equation}
	Then every sublevel set of $W$ is compact.
\end{lemma}
\begin{proof}
	Since all sublevel sets of lower semicontinuous functions are closed (in this case relative to $\GLp(2)$), we only need to show that for each $c\in\R$, the corresponding sublevel set
	\[
		S_c \colonequals \{F\in\R^{2\times2} \setvert \det F > 0\,,\; W(F) \leq c\}
	\]
	is bounded and that $\dist(S_c,\partial\GLp(2))>0$. In order to simplify the computations, we will consider the distance on $\R^{2\times2}$ with respect to the operator norm, given by the largest singular value $\opnorm{X}=\lambdamax(X)$ of a matrix $X\in\R^{2\times2}$. First note that
	\begin{equation}\label{eq:opnormKdet}
		\opnorm{F}^2 = \frac{\lambdamax}{\lambdamin}\cdot\lambdamax\lambdamin = K(F)\cdot\det(F)
	\end{equation}
	and that the lower semicontinuity of $\hhat$ and $f$, together with \eqref{eq:compactnessGrowthConditions}, ensures that both $\hhat$ and $f$ are bounded below by some $d\in\R$.
	
	Now, let $c\in\R$. Then due to \eqref{eq:compactnessGrowthConditions}, there exists $r\geq1$ such that $\hhat(t)>c-d$ and $f(t)>c-d$ for any $t>r$. Since $\opnorm{F}^2>r^2$ implies that either $K(F)>r$ or $\det(F)>r$ by virtue of \eqref{eq:opnormKdet},
	\[
		W(F) = \hhat(K(F)) + f(\det(F)) > c-d+d = c
	\]
	for any $F\in\GLp(2)$ with $\opnorm{F}>r$. Thus $S_c$ is bounded.
	
	In order to establish a lower bound for $\dist(S_c,\partial\GLp(2))$, we first observe that, according to the Eckart-Young-Mirsky Theorem \cite{mirsky1960symmetric},\footnote{%
		Recall that the distance is taken with respect to the operator norm. Equality \eqref{eq:rankOneApproximation} can also easily be shown directly: on the one hand, for any $X\in\R^{2\times2}$ with $\rank(X)\leq1$ there exists $\xi\in\R^2$ with $X\xi=0$ and $\norm{\xi}=1$, thus $\opnorm{F-X}\geq\norm{(F-X)\xi}=\norm{F\xi}\geq\lambdamin$; on the other hand, $\opnorm{Q_1\diag(\lambdamax,\lambdamin)Q_2-Q_1\diag(\lambdamax,0)Q_2}=\lambdamin$ for $Q_1,Q_2\in\SO(2)$.%
	}
	\begin{align}
		\dist(F,\partial\GLp(2))
		&= \dist(F,\{X\in\R^{2\times2}\setvert\det(X)=0\}) \nonumber\\
		&= \dist(F,\{X\in\R^{2\times2}\setvert\rank(X)\leq1\})
		= \lambdamin(F)\,. \label{eq:rankOneApproximation}
	\end{align}
	Due to \eqref{eq:compactnessGrowthConditions}, there exists $r\leq1$ such that $\hhat(s)>c-d$ and $f(t)>c-d$ for all $t<r$ and all $s>\frac1r$. Then for any $F\in\GLp(2)$ with $\lambdamin<r$,
	\[
		r^2 > \lambdamin^2 = \frac{\lambdamax\lambdamin}{\frac{\lambdamax}{\lambdamin}} = \frac{\det(F)}{K(F)}
	\]
	and thus either $\det(F)>r$ or $K(F)<\frac1r$. Again, $W(F)>c$ in either case; therefore, $F\notin S_c$ for any $F\in\GLp(2)$ with $\dist(F,\partial\GLp(2)) = \lambdamin(F) < r$, which shows that $\dist(S_c,\partial\GLp(2))\geq r >0$.
\end{proof}

We also require a criterion for the sublevel sets of $W$ to be connected. Recall from Definition \ref{definition:q-Convexity} that a function is called q-convex if all its sublevel sets are convex.

\begin{lemma}\label{lemma:connectedness}
	Let $W\col\GLp(2)\to\R$ be of the form \eqref{eq:volIsoSplitGeneralFormSingularValues} such that $\hhat\col[1,\infty)\to \R$ is monotone and $f$ is q-convex. Then every sublevel set of $W$ is connected.
\end{lemma}
\begin{proof}
\newcommand{\pows}{^{\raisebox{-.245em}{\!\text{\normalsize$s$}}}}
	In order to show that the sublevel set $S_c$ is (path) connected for any $c\in\R$ under the stated conditions, we will explicitly construct a curve connecting arbitrary $F,\Ftilde\in S_c$. The construction will be split into four parts, with the first and last one describing a continuous, orthogonal basis change. The second part will be constructed such that the linear distortion $K$ is decreasing and the determinant is constant along the curve, while for the third part, the linear distortion $K$ is kept constant. This will allow us to utilize the monotonicity of the isochoric part $\hhat$ and the q-convexity of the volumetric part $f$, respectively, to show that the curve remains in the sublevel set (cf.\ Figure \ref{fig:curveWithinSublevelsCounterexample}).
	
	For $c\in\R$, let $F,\Ftilde\in S_c$ and assume without loss of generality that $K(\Ftilde) \leq K(F)$. First, choose $Q_1,Q_2,\Qtilde_1,\Qtilde_2\in\SO(2)$ such that $\diag(\lambda_1,\lambda_2)\colonequals Q_1FQ_2$ and $\diag(\lambdatilde_1,\lambdatilde_2)\colonequals \Qtilde_1\Ftilde\Qtilde_2$ are diagonal with $\lambda_1\geq\lambda_2>0$ and $\lambdatilde_1\geq\lambdatilde_2>0$. Let
	\begin{equation}
		\curve_1\col[0,1]\to\GLp(2)\,,\qquad \curve_1(s) = Q_1^s FQ_2^s\,;
	\end{equation}
	here, $Q^s\colonequals \exp(s\log Q)$ for $Q\in\SO(2)$, where $\log$ denotes the principal matrix logarithm on $\SO(2)$ or, more explicitly,
	\begin{equation}
		Q^s = \matr{\cos(s\.\alpha)&\sin(s\.\alpha)\\-\sin(s\.\alpha)&\cos(s\.\alpha)} \qquad\text{for}\qquad Q = \matr{\cos(\alpha)&\sin(\alpha)\\-\sin(\alpha)&\cos(\alpha)} \quad\text{with }\;\alpha\in(-\pi,\pi]\,.
	\end{equation}
	In particular, $\curve_1$ is continuous with $\curve_1(0)=F$ and $\curve_1(1)=\diag(\lambda_1,\lambda_2)$. Furthermore, $s\mapsto W(\curve_1(s))$ is constant due to the objectivity and isotropy of $W$ and thus, in particular, $\curve_1(s)\in S_c$ for all $s\in[0,1]$.
	
	Now, for $\mu_1=\frac{\sqrt{\lambda_1\.\lambda_2}}{\sqrt{\lambdatilde_1\.\lambdatilde_2}}\,\lambdatilde_1$ and $\mu_2=\frac{\sqrt{\lambda_1\.\lambda_2}}{\sqrt{\lambdatilde_1\.\lambdatilde_2}}\,\lambdatilde_2$, let
	\begin{equation}
		\curve_2\col[0,1]\to\GLp(2)\,,\qquad \curve_2(s) = \diag(\lambda_1^{1-s}\.\mu_1^s,\lambda_2^{1-s}\.\mu_2^s)\,.
	\end{equation}
	Then $\curve_2$ is continuous with $\curve_2(0)=\diag(\lambda_1,\lambda_2)=\curve_1(1)$ and $\curve_2(1)=\diag(\mu_1,\mu_2)$. Furthermore,
	\[
		\det(\curve_2(s)) = \lambda_1^{1-s}\.\mu_1^s\.\lambda_2^{1-s}\.\mu_2^s = \lambda_1\lambda_2\cdot \biggl(\frac{\mu_1\mu_2}{\lambda_1\lambda_2}\biggr)\pows = \lambda_1\lambda_2\cdot \left(\frac{\frac{\sqrt{\lambda_1\.\lambda_2}}{\sqrt{\lambdatilde_1\.\lambdatilde_2}}\,\lambdatilde_1\cdot\frac{\sqrt{\lambda_1\.\lambda_2}}{\sqrt{\lambdatilde_1\.\lambdatilde_2}}\,\lambdatilde_2}{\lambda_1\lambda_2}\right)\pows = \lambda_1\lambda_2
	\]
	is independent of $s$ and, since $\frac{K(\Ftilde)}{K(F)}\leq1$ by assumption, the mapping
	\[
		s\mapsto K(\curve_2(s)) = \frac{\lambda_1^{1-s}\.\mu_1^s}{\lambda_2^{1-s}\.\mu_2^s} = \frac{\lambda_1}{\lambda_2}\cdot \biggl(\frac{\lambda_2\.\mu_1}{\lambda_1\.\mu_2}\biggr)\pows = \frac{\lambda_1}{\lambda_2}\cdot \left(\frac{K(\Ftilde)}{K(F)}\right)\pows
	\]
	is nonincreasing on $[0,1]$. Therefore, the assumed monotonicity of the function $\hhat$ implies that the mapping $s\mapsto W(\curve_2(s))=h(K(\curve_2(s)))+f(\det(\curve_2(s)))$ is nonincreasing, hence $\curve_2(s)\in S_c$ for all $s\in[0,1]$.
	
	Next, let
	\[
		\curve_3\col[0,1]\to\GLp(2)\,,\qquad \curve_3(s) = \left(\frac{\sqrt{\lambdatilde_1\.\lambdatilde_2}}{\sqrt{\lambda_1\.\lambda_2}}\right)\pows\cdot\,\diag(\mu_1,\mu_2)\,.
	\]
	Then $\curve_3$ is continuous with $\curve_3(0)=\diag(\mu_1,\mu_2)=\curve_2(1)$ and $\curve_3(1)=\diag(\lambdatilde_1,\lambdatilde_2)$.
	We observe that $K$ is constant along $\curve_3$ and that
	\[
		\min\{\det(\curve_3(0))\,, \det(\curve_3(1))\}
		\leq\det(\curve_3(s))
		\leq \max\{\det(\curve_3(0))\,, \det(\curve_3(1))\}
	\]
	for all $s\in[0,1]$. Furthermore, the volumetric part $f\col(0,\infty)\to\R$ is q-convex by assumption which, in particular \cite{crouzeix1982criteria}, implies that $f(t)\leq \max\{f(a),\,f(b)\}$ for any $0<a\leq t \leq b$. We therefore find
	\[
		f(\det (\curve_3(s))) \leq \max\{f(\det(\curve_3(0) ))\,,\;f(\det(\curve_3(1)))\}
	\]
	and thus, since $s\mapsto K(\curve_3(s))$ is constant (i.e.\ $K(\curve_3(s))=K(\curve_3(0))=K(\curve_3(1))$ for all $s\in[0,1]$),
	\begin{align*}
		W(\curve_3(s)) &= f(\det (\curve_3(s))) + \hhat(K(\curve_3(s)))\\
		&\leq \max\{f(\det(\curve_3(0))) + \hhat(K(\curve_3(s)))\,,\; f(\det(\curve_3(1))) + \hhat(K(\curve_3(s)))\}\\
		&= \max\{f(\det(\curve_3(0))) + \hhat(K(\curve_3(0)))\,,\; f(\det(\curve_3(1))) + \hhat(K(\curve_3(1)))\}\\
		&= \max\{W(\curve_3(0))\,,\; W(\curve_3(1))\}
		\;\leq\; c
	\end{align*}
	for all $s\in[0,1]$.
	
	Finally, the curve
	\[
		\curve_4\col[0,1]\to\GLp(2)\,,\qquad \curve_4(s) = (\Qtilde_1^T)^s \diag(\lambdatilde_1,\lambdatilde_2)(\Qtilde_2^T)^s
	\]
	continuously connects $\diag(\lambdatilde_1,\lambdatilde_2)$ with $\Ftilde$ such that $W$ is constant along $\curve_4$. Therefore, the combined curve
	\[
		\curve\col[0,4]\to\GLp(2)\,,\qquad \curve(s) = \curve_i(s)
		\quad\text{ for }\; s\in[i-1,i]\,,
		\;\;i\in\{1,2,3,4\}
	\]
	is (well defined and) continuous with $\curve(0)=F$, $\curve(4)=\Ftilde$ and $\curve(s)\in S_c$ for all $s\in[0,4]$.
\end{proof}

\begin{figure}[h!]
	\begin{center}
		\begin{tikzpicture}
		\def\Dconstant{3.5}
		\def\Kconstant{.735}
			\begin{axis}[
				axis x line=middle,axis y line=middle,
				x label style={at={(current axis.right of origin)},anchor=north, below},
				xlabel={$\lambda_1$}, ylabel={$\lambda_2$},
				xmin=-0.07, xmax=4.9,
				ymin=-0.07, ymax=4.9,
				xtick={3.43},
				ytick={3.185},
				xticklabels={$\mu_1$},
				yticklabels={$\mu_2$},
				hide obscured x ticks=false,
				width=.49\linewidth,
				height=.49\linewidth
			]
				\addplot[domain=0:4.41,dotted,thick] {x} node[pos=.63, above left] {$\lambda_1=\lambda_2$};
				\addplot[domain=.343:sqrt(\Dconstant/\Kconstant),blue,dashed,thick] {\Dconstant/x}
					node[pos=.7, above right] {$\lambda_1\lambda_2\equiv \det(F)$};
				\addplot[domain=sqrt(\Dconstant/\Kconstant):3.43,<-,ultra thick,blue] {\Dconstant/x}
					node[pos=1.001,draw,circle,fill, inner sep=0.147em]{}
					node[pos=1.001, above right] {$F$}
					node[pos=.49, below left] {$\curve_2$};
				\addplot[domain=3.43:4.9,blue,dashed,thick] {\Dconstant/x};
				\addplot[domain=0:sqrt(\Dconstant/\Kconstant),red,dashed,thick] {\Kconstant*x}
					node[pos=.392, below right] {$\frac{\lambda_1}{\lambda_2}\equiv K(\Ftilde)$};
				\addplot[domain=sqrt(\Dconstant/\Kconstant):4.312,->,ultra thick,red] {\Kconstant*x}
					node[pos=1.029,draw,circle,fill, inner sep=0.147em]{}
					node[pos=1.029, below right] {$\Ftilde$}
					node[pos=.49, below right] {$\curve_3$};
				\addplot[domain=4.312:4.9,red,dashed,thick] {\Kconstant*x};
			\end{axis}
		\end{tikzpicture}
		\hfill
		\begin{tikzpicture}
				\def\Kconstant{1.26}
		\def\Dconstant{3.5}
			\begin{axis}[
				axis x line=middle,axis y line=middle,
				x label style={at={(current axis.right of origin)},anchor=north, below},
				xlabel={$s$}, ylabel={$W(\curve(s))$},
				xmin=-0.07, xmax=4.13,
				ymin=-.07, ymax=.847,
				xtick={1,2,3,4},
				ytick=\empty,
				hide obscured x ticks=false,
				width=.588\linewidth,
				height=.49\linewidth
			]
				\addplot[domain=0:1.001, thick] {.49}
					node[pos=.49,below] {$\curve_1$};
				\addplot[domain=1.001:2.002, thick,blue] {.49-0.343*(x-1.001)^2}
					node[pos=.63, below left] {$\curve_2$};
				\addplot[domain=2.002:2.506, thick,red] {.49-0.343*(1.001)^2-.147*(x-2.002)};
				\addplot[domain=2.506:3.003, thick,red] {.49-0.343*(1.001)^2-.147*(.504)+3.43*((x-2.1)^(1/7)-(2.506-2.1)^(1/7))}
					node[pos=.49, above left] {$\curve_3$};
				\addplot[domain=3.003:3.997, thick] {.49-0.343*(1.001)^2-.147*(.504)+3.43*((3.003-2.1)^(1/7)-(2.506-2.1)^(1/7))}
					node[pos=.49, below] {$\curve_4$};
			\end{axis}
		\end{tikzpicture}
		\vspace*{-1.47em}
	\par\end{center}
	\caption{%
	\label{fig:curveWithinSublevelsCounterexample}
		Along $\curve_2$, the linear distortion $K=\frac{\lambda_1}{\lambda_2}$ is decreasing and the determinant is constant, so the energy is nonincreasing; along $\curve_3$, $K$ is constant, so the curve does not leave the sublevel set due to the q-convexity of the energy with respect to the determinant.%
	}
\end{figure}
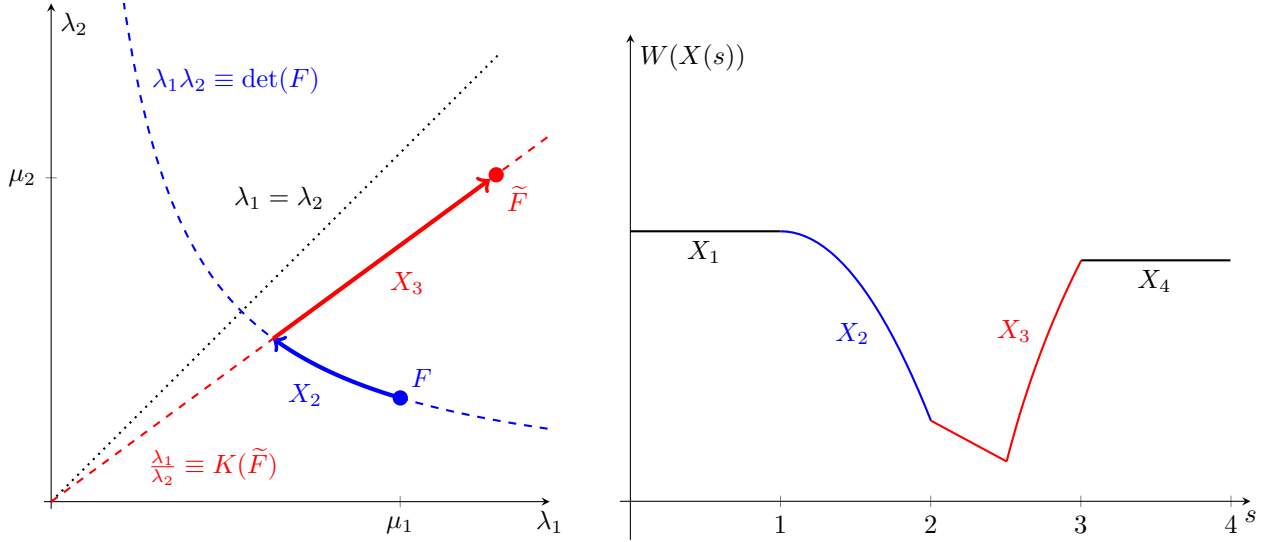

Of course, in order to show that $\Wex$ given by \eqref{eq:counterexampleDefinition} is a counterexample to Conjecture \ref{conjecture:mielke}, the rank-one convexity of $\Wex$ must be established as well. The following criterion, which is based on an earlier result by Knowles and Sternberg \cite{knowles1976failure,knowles1978failure} for twice differentiable functions (cf.\ Appendix \ref{section:knowlesSternberg}) and which is applicable to volumetric-isochorically split energies in the planar case, has been the subject of a recent contribution \cite{agn_voss2019volIsLog}.

\begin{lemma}[\cite{agn_voss2019volIsLog}]\label{lemma:rankOneConvexity}
	For $f,\hunordered\in C^2((0,\infty))$ with $\hunordered\left(\frac1t\right)=\hunordered(t)$, let $W(F)=\hunordered\bigl(\frac{\lambda_1}{\lambda_2}\bigr)+f(\lambda_1\lambda_2)$ for all $F\in\GLp(2)$ with singular values $\lambda_1,\lambda_2$. Then $W$ is rank-one convex on $\GLp(2)$ if and only if
	\begin{itemize}
		\item[i)] \quad$\displaystyle h_0+f_0\geq 0$,
		\item[ii)] \quad$\displaystyle \hunordered'(t)\geq 0\qquad\text{for all }\; t\geq 1$,
		\item[iii)]\quad$\displaystyle \frac{2t}{t-1}\.\hunordered'(t)-t^2\.\hunordered''(t)+f_0\geq 0$\quad or\quad $\displaystyle a(t)+\left[b(t)-c(t)\right]f_0\geq 0\qquad\text{for all }\; t\in(0,\infty)\,,\;t\neq 1$,
		\item[iv)]\quad$\displaystyle \frac{2t}{t+1}\.\hunordered'(t)+t^2\.\hunordered''(t)-f_0\geq 0\quad\text{or}\quad a(t)+\left[b(t)+c(t)\right]f_0\geq 0\qquad\text{for all }\; t\in(0,\infty)$,
	\end{itemize}
	where
	\begin{align*}
		a(t)&=t^2(t^2-1)\.h'(t)h''(t)-2t\.h'(t)^2\,,\qquad b(t)=\left(t^2+3\right)h'(t)+2t\.(t^2+1)\.h''(t)\,,\\
		c(t)&=4t\left(h'(t)+t\.h''(t)\right)
	\end{align*}
	and
	\[
		h_0=\inf_{t\in(0,\infty)}t^2\.\hunordered''(t)\,,\qquad f_0=\inf_{t\in(0,\infty)}t^2\.f''(t)\,.\directqed
	\]
\end{lemma}
Note that the function $\hunordered\col(0,\infty)\to\R$, which represents the isochoric part of $W$ in terms of the ratio $\frac{\lambda_1}{\lambda_2}$ of the (not necessarily ordered) singular values, can be obtained from $\hhat\col[1,\infty)\to\R$ as given in \eqref{eq:volIsoSplitGeneralFormSingularValues} by setting
\begin{equation}\label{ht}
	\hunordered(t) \colonequals
	\begin{cases}
		\hhat(t) &: t\geq1\,,\\
		\hhat\left(\frac1t\right) &: t<1\,.
	\end{cases}
\end{equation}

\begin{remark}
	The energy $\Wex$ can be written in the form
	\[
		\Wex(F)=h\left(\frac{\lambdamax}{\lambdamin}\right)+\hunordered\left(\frac{1}{\lambdamax\lambdamin}\right)\,,
	\]
	where $\hunordered\col(0,\infty)\to\R$ is the function related to $\hhat\col[1,\infty)\to\R$ with $\hhat(t)=t-\log(t)$ by \eqref{ht}.
\end{remark}
%
%
%
%
%
\subsection{Main properties of the counterexample}%
\label{section:assemblingTheProof}

We can now prove Proposition \ref{prop:conjectureFalse}, i.e.\ demonstrate that Conjecture \ref{conjecture:mielke} does not hold, by showing that $\Wex$ is indeed a counterexample.
\begin{proposition}
	The function
	\begin{equation}\label{eq:counterexampleRepetition}
		\Wex(F) = \frac{\lambdamax}{\lambdamin} - \log\left(\frac{\lambdamax}{\lambdamin}\right) + \log(\lambdamax\lambdamin) + \frac{1}{\lambdamax\lambdamin} = \hhat(K(F)) + f(\det F)
	\end{equation}
	with
	\[
		\hhat(t) = t-\log(t) \qquad\text{and}\qquad f(t)=\hunordered\left(\frac{1}{t}\right)=\log(t)+\frac1t\,,
	\]
	where $\lambdamax\geq\lambdamin>0$ are the singular values of $F\in\GLp(2)$, is rank-one convex and isotropic with compact and connected sublevel sets, but not polyconvex.
\end{proposition}
\begin{corollary*}[Proposition \ref{prop:conjectureFalse}]
	Conjecture \ref{conjecture:mielke} does not hold.
\end{corollary*}
\begin{proof}
	The function $\Wex$ is obviously objective and isotropic. Since $\hhat$ and $f$ are lower semicontinuous with $\lim_{t\to\infty}\hhat(t)=\lim_{t\to\infty}f(t)=\lim_{t\to0}f(t)=\infty$, Lemma \ref{lemma:compactness} shows that the sublevel sets of $\Wex$ are compact. Furthermore, $\hhat$ is monotone increasing on $[1,\infty)$ and $f$ is q-convex on $(0,\infty)$, thus the sublevel sets are also connected due to Lemma \ref{lemma:connectedness}.
	
	\bigskip
	
	It remains to show that $\Wex$ is rank-one convex\footnote{%
		An alternative proof of the rank-one convexity of $\Wex$, based directly on the classical Knowles-Sternberg criterion, is given in Appendix \ref{section:knowlesSternberg}.%
	}
	but not polyconvex. In order to prove the rank-one convexity, we will apply Lemma \ref{lemma:rankOneConvexity}. Since\footnote{%
		Note that $\hunordered$ is two-times continuously differentiable.%
	}
	\[
		\hunordered(t)
		\;=\;
		\begin{cases}
			\hhat(t) &: t\geq1\\
			\hhat\left(\frac1t\right) &: t<1
		\end{cases}
		\quad=\quad
		\begin{cases}
			t-\log(t) &: t\geq1\\
			\frac1t+\log(t) &: t<1
		\end{cases}
	\]
	for $t>0$, we find
	\begin{align}
		h_0 &= \inf_{t\in(0,\infty)}t^2\.\hunordered''(t) = \min \left\{ \inf_{t\in[1,\infty)}t^2\.\hunordered''(t)\,,\; \inf_{t\in(0,1)}t^2\.\hunordered''(t) \right\} \nonumber\\
		&= \min \left\{ \inf_{t\in[1,\infty)} t^2\cdot\frac{1}{t^2}\,,\; \inf_{t\in(0,1)} t^2\cdot \left( \frac{2}{t^3} -\frac{1}{t^2} \right) \right\}
		= \min \left\{ 1\,,\; \inf_{t\in(0,1)} \left( \frac{2}{t} -1 \right) \right\}
		= 1
	\end{align}
	as well as
	\begin{align*}
		f_0 &= \inf_{t\in(0,\infty)}t^2\.f''(t) = \inf_{t\in(0,\infty)} t^2\cdot \left( -\frac{1}{t^2} + \frac{2}{t^3} \right) = \inf_{t\in(0,\infty)} -1+\frac{2}{t} = -1\,.
	\end{align*}
	Then $h_0+f_0=0$, thus condition i) in Lemma \ref{lemma:rankOneConvexity} is satisfied. Since $\hunordered'(t)=\hhat'(t)=1-\frac1t\geq0$ for all $t\geq1$, condition ii) is fulfilled as well. Next, we compute
	\begin{align*}
		\frac{2t}{t-1}\,\hunordered'(t)-t^2\.\hunordered''(t)+f_0
		&= \frac{2t}{t-1}\,\hhat'(t)-t^2\.\hhat''(t)+f_0\\
		&= \frac{2t}{t-1}\cdot \left( 1-\frac1t \right) - t^2\cdot \frac{1}{t^2} - 1
		= \frac{2t}{t-1}\cdot \frac{t-1}{t} - 2
		\;=\; 0
	\end{align*}
	for $t\geq1$ as well as
	\begin{align}
		\frac{2t}{t-1}\,\hunordered'(t)-t^2\.\hunordered''(t)+f_0
		&= \frac{2t}{t-1}\cdot \left( -\frac{1}{t^2}+\frac1t \right) - t^2\cdot \left( \frac{2}{t^3} - \frac{1}{t^2} \right) - 1 \nonumber\\
		&= \frac{2t}{t-1}\cdot \frac{t-1}{t^2} - t^2\cdot \frac{2-t}{t^3} - 1
		= \frac{2}{t} - \frac{2-t}{t} - 1
		\;=\; 0
	\end{align}
	for $t<1$, which shows that condition iii) is also satisfied. Finally, for condition iv), we find
	\[
		\frac{2t}{t+1}\.\hunordered'(t)+t^2\.\hunordered''(t)-f_0
		= \frac{2t}{t+1}\cdot\frac{t-1}{t} + t^2\cdot\frac{1}{t^2} + 1
		= 2\, \left( \frac{t-1}{t+1}+1 \right)
		\geq 0
	\]
	for all $t\geq1$ and
	\begin{align*}
		\frac{2t}{t+1}\.\hunordered'(t)+t^2\.\hunordered''(t)-f_0
		&= \frac{2t}{t+1}\cdot\frac{t-1}{t^2} + t^2\cdot\frac{2-t}{t^3} + 1\\
		&= \frac{2t-2}{t\.(t+1)} + \frac{2-t}{t} + 1
		= \frac{2t-2}{t\.(t+1)} + \frac2t
		= \frac{2t-2 + 2\.(t+1)}{t\.(t+1)}
		= \frac{4}{t+1}
		\;\geq\; 0
	\end{align*}
	for all $t<1$. Therefore, according to Lemma \ref{lemma:rankOneConvexity}, the function $\Wex$ is rank-one convex.
	
	\bigskip
	
	Finally, in order to show that $\Wex$ is not polyconvex, we will apply \v{S}ilhav\'{y}'s result given in Proposition \ref{polycrit}. The representation $\ghat$ of $\Wex$ in terms of ordered singular values is given by
	\begin{equation}\label{eq:counterexampleSingularValueRepresentation}
		\Wex(F) = \ghat(\svmap(F))
		\qquad\text{with}\qquad
		\ghat(\lambdahat_1,\lambdahat_2) = \frac{\lambdahat_1}{\lambdahat_2} - \log\left(\frac{\lambdahat_1}{\lambdahat_2}\right) + \log(\lambdahat_1\lambdahat_2) + \frac{1}{\lambdahat_1\lambdahat_2}\,.
	\end{equation}
	We compute
	\[
		\frac{\partial \ghat}{\partial \lambdahat_1}=\frac{\lambdahat_1^2-1}{\lambdahat_1^2 \lambdahat_2},\qquad 
		\frac{\partial \ghat}{\partial \lambdahat_2}=-\frac{\lambdahat_1^2-2 \lambdahat_1 \lambdahat_2+1}{\lambdahat_1 \lambdahat_2^2}
	\]
	and choose
	\[
		\gamma_1 = e^4, \qquad \gamma_2 = e^3, \qquad \nu_1 = e, \qquad \nu_2 = 1\,.
	\]
	Then
	\begin{alignat}{2}
		\label{eq:nonPolyValues1}
		\frac{\partial \ghat}{\partial\lambdahat_1}(e^4,e^3) &= \frac{e^8-1}{e^{11}}\,,
		\qquad&
		\frac{\partial \ghat}{\partial\lambdahat_2}(e^4,e^3) &= -\frac{e^7 (e-2)+1}{e^{10}}\,,
		\\
		\notag
		-\frac{\frac{\partial \ghat}{\partial\lambdahat_1}(e^4,e^3)-\frac{\partial \ghat}{\partial\lambdahat_2}(e^4,e^3)}{e^4-e^3} &= -\frac{1+e^8}{e^{14}}\,,
		\qquad&
		\frac{\frac{\partial \ghat}{\partial\lambdahat_1}(e^4,e^3)+\frac{\partial \ghat}{\partial\lambdahat_2}(e^4,e^3)}{e^4+e^3} &= -\frac{1+e-3 e^8+e^9}{e^{14} (1+e)}
	\end{alignat}
	and therefore, if $\Wex$ were polyconvex,
	\begin{equation}\label{eq:nonPolyMielkeCondition}
		\begin{aligned}
			&\exists\;c\in \left[ -\frac{1+e^8}{e^{14}}\,,\; -\frac{1+e-3 e^8+e^9}{e^{14} (1+e)} \right]
			\quad \text{ such that}
			\\
			&\quad \ghat(e,1)\geq \ghat(e^4,e^3)+\frac{\partial \ghat}{\partial\lambdahat_1}(e^4,e^3)\cdot(e-e^4)
			+\frac{\partial \ghat}{\partial\lambdahat_2}(e^4,e^3)\cdot(1-e^3)+c\, (e-e^4)\,(1-e^3)
		\end{aligned}
	\end{equation}
	according to Proposition \ref{polycrit}. Since
	\[
		\ghat(e^4,e^3) = e+6+\frac{1}{e^7}
		\qquad\text{and}\qquad
		\ghat(e,1) = \frac{1}{e}+e\,,
	\]
	combining \eqref{eq:nonPolyMielkeCondition} with \eqref{eq:nonPolyValues1} yields that polyconvexity of $\Wex$ would imply
	\[
		\exists\; c\in \left[-\frac{1+e^8}{e^{14}},-\frac{1+e-3 e^8+e^9}{e^{14} (1+e)}\right]
		\qquad \text{such that}\quad
		c\leq \frac{2-3 e^3-2 e^7+e^9-4 e^{10}}{e^{11} \left(e^3-1\right)^2}\,,
	\]
	in contradiction to
	\[
		-0.00377147 \approx \frac{2-3 e^3-2 e^7+e^9-4 e^{10}}{e^{11} \left(e^3-1\right)^2}<-\frac{1+e^8}{e^{14}}\approx-0.00247958\,.
	\]
	Therefore, the energy $\Wex$ cannot be polyconvex. 
\end{proof}

\section{Conclusion}

We have provided yet another example of a rank-one convex isotropic energy function on $\GLp(2)$ which is not polyconvex. In contrast to classical examples encountered in the literature (cf.\ Section \ref{section:previousExamples}), however, the energy $\Wex$ given in \eqref{eq:counterexampleDefinition} satisfies the growth condition $\Wex(F)\to\infty$ for $\det F \to 0$. Moreover, all sublevel sets of $\Wex$ are compact as well as connected, thus $\Wex$ serves as a counterexample to an earlier conjecture by Mielke stating that rank-one convexity of an isotropic function with these properties implies its polyconvexity. As for many other examples of non-polyconvex planar energy functions, it is currently unknown whether the rank-one convex function $\Wex$ is quasiconvex as well. In particular, it remains possible that the conditions of Conjecture \ref{conjecture:mielke} are indeed sufficient to ensure the quasiconvexity of rank-one convex planar functions in general.
%
%
%
%
%
\subsubsection*{Acknowledgements}
The work of I.D.\ Ghiba has been supported by a grant of the Romanian Ministry of Research and Innovation, CNCS--UEFISCDI, project number PN-III-P1-1.1-TE-2019-0397, within PNCDI III.

\section{References}
\footnotesize
\printbibliography[heading=none]

\begin{appendix}

\section{Connectedness of sublevels for Aubert's energy}
\label{appendix:aubertConnectedness}

Although Aubert's example \cite{aubert1987counterexample}
\begin{align*}
	\WA\col\GLp(2)\to\R\,,\quad \WA(F) &= \frac13(\lambda_1^4+\lambda_2^4) + \frac12\lambda_1^2\lambda_2^2 - \frac23(\lambda_1^3\lambda_2+\lambda_1\lambda_2^3)=\frac13\.\norm{F}^4-\frac16\.(\det F)^2-\frac23\.\det F\cdot\norm{F}^2
\end{align*}
of a rank-one convex, non-polyconvex energy function is not suitable as a counterexample to Conjecture \ref{conjecture:mielke} since the sublevel sets $S_c$ are not compact, we will show in the following that every set $S_c$ is indeed connected.

For $c\in\R$, let $F,\Ftilde\in S_c$. As in the proof of Lemma \ref{lemma:connectedness}, we will explicitly construct a continuous curve contained in $S_c$ which connects $F$ and $\Ftilde$. Without loss of generality,\footnote{%
	Since $\WA$ is objective and isotropic, arbitrary $F$ and $\Ftilde$ can be continuously connected to diagonal matrices by curves within level sets of $\WA$ as in the proof of Lemma \ref{lemma:connectedness}.%
}
assume that $F=\diag(\lambda_1,\lambda_2)$ and $\Ftilde=\diag(\lambdatilde_1,\lambdatilde_2)$ are diagonal with $\lambda_1\geq\lambda_2$, $\lambdatilde_1\geq\lambdatilde_2$ and $\lambdatilde_1\geq\lambda_1$. Then the curve
\begin{equation}
	\curve_1\col[\lambda_2,\lambda_1]\to\GLp(2)\,,\qquad \curve_1(s) = \diag(\lambda_1,s)
\end{equation}
is continuous with $\curve_1(\lambda_2)=F$ and, for all $s\in[\lambda_2,\lambda_1]$,
\begin{align}
	\dds\;\WA(\curve_1(s))
	&= \frac43\,s^3 - 2\lambda_1\.s^2 + \lambda_1^2\.s - \frac23\,\lambda_1^3 \leq \frac13\,s^3 + \lambda_1\.s^2 - 2\lambda_1\.s^2 + \lambda_1^2\.s - \frac23\,\lambda_1^3\nonumber\\
	&= \frac13\,s^3 + \lambda_1\.s\,(\lambda_1-s) - \frac23\,\lambda_1^3 \leq \frac13\,s^3 + \frac12\,(\lambda_1^2+s^2)\,(\lambda_1-s) - \frac23\,\lambda_1^3 \label{eq:usedCauchyInequalityFirst}\\
	&= \frac13\,s^3 + \frac12\,(\lambda_1^3+\lambda_1\.s^2-\lambda_1^2\.s-s^3) - \frac23\,\lambda_1^3
	= -\frac16\,s^3 + \frac12\,\lambda_1\.s\,(s-\lambda_1) - \frac16\,\lambda_1^3 \nonumber
	\;<\;0\,, \nonumber
\end{align}
where \eqref{eq:usedCauchyInequalityFirst} holds due to \emph{Cauchy's inequality},
\begin{equation}\label{eq:cauchyInequality}
	x\,y\leq k\,x^2+\frac{y^2}{4\,k} \qquad\text{for all }\; x,y,k>0\,,
\end{equation}
with $x=\lambda_1$, $y=s$ and $k=\frac12$. Therefore, the mapping $s\mapsto \WA(\curve_1(s))$ is monotone decreasing, which implies $\WA(\curve_1(s))\leq\WA(\curve_1(\lambda_2))=\WA(F)\leq c$ for all $s\in[\lambda_2,\lambda_1]$.

Now, let
\[
	\curve_2\col[\lambda_1,\lambdatilde_1]\to\GLp(2)\,,\qquad \curve_2(s) = \diag(s,s)\,.
\]
Then $\curve_2$ is continuous with $\curve_2(\lambda_1)=\curve_1(\lambda_1)$. Furthermore, $\displaystyle W(\curve_2(s)) = -\frac{s^4}{6}\,,$ thus $\WA$ is decreasing along $\curve_2$ as well. Finally, for the continuous curve
\[
	\curve_3\col[0,\lambdatilde_1-\lambdatilde_2]\to\GLp(2)\,,\quad \curve_3(s) = \diag(\lambdatilde_1,\lambdatilde_1-s)\,,
\]
we find $\curve_3(0)=\curve_2(\lambdatilde_1)$ and $\curve_3(\lambdatilde_1-\lambdatilde_2)=\diag(\lambdatilde_1,\lambdatilde_2)=\Ftilde$ as well as
\begin{align}
	\dds\;\WA(\curve_3(s))
	&= \dds\;\left( \frac13(\lambdatilde_1^4+(\lambdatilde_1-s)^4) + \frac12\lambdatilde_1^2\.(\lambdatilde_1-s)^2 - \frac23(\lambdatilde_1^3\.(\lambdatilde_1-s)+\lambdatilde_1\.(\lambdatilde_1-s)^3) \right) \nonumber\\
	&= \frac13\,\lambdatilde_1^3 + \lambdatilde_1^2\.s - 2\lambdatilde_1\.s^2 + \frac43\.s^3 \geq \frac13\,\lambdatilde_1^3 - \lambdatilde_1\.s^2 + \frac43\.s^3 \nonumber\\
	&\geq \frac13\,\lambdatilde_1^3 - s\.\left( \frac13\.\lambdatilde^2 + \frac34\.s^2 \right) + \frac43\.s^3
	= \frac13\,\lambdatilde^2\.(\lambdatilde_1-s) + \left( \frac43-\frac34 \right)\.s^3
	\;>\; 0\,, \label{eq:usedCauchyInequalitySecond} 
\end{align}
where Cauchy's inequality \eqref{eq:cauchyInequality} was again employed in \eqref{eq:usedCauchyInequalitySecond} with $x=\lambdatilde_1$, $y=s$ and $k=\frac13$. Thus $\WA$ is nondecreasing along $\curve_3$, which implies
\[
	\WA(\curve_3(s)) \leq \WA(\curve_3(\lambdatilde_1-\lambdatilde_2)) = \WA(\Ftilde) \leq c
\]
and hence $\curve_3(s)\in S_c$ for all $s\in[0,\lambdatilde_1-\lambdatilde_2]$. By concatenating $\curve_1$, $\curve_2$ and $\curve_3$, we therefore obtain a continuous curve contained in $S_c$ which connects $F$ and $\Ftilde$, cf.\ Figure \ref{fig:curveWithinSublevelsAubert}.

\begin{figure}
	\begin{center}
		\begin{tikzpicture}
			\begin{axis}[
				axis x line=middle,axis y line=middle,
				x label style={at={(current axis.right of origin)},anchor=north, below},
				xlabel={$\lambda_1$}, ylabel={$\lambda_2$},
				xmin=-0.07, xmax=4.9,
				ymin=-0.07, ymax=4.9,
				xtick=\empty,
				ytick=\empty,
				hide obscured x ticks=false,
				width=.49\linewidth,
				height=.39\linewidth
			]
				\addplot[domain=0:4.41,dotted,thick] {x} node[pos=.49, above left] {$\lambda_1=\lambda_2$};
				\addplot[domain=.49:1.47,samples=2,->,ultra thick,blue] ({1.47},{x})
					node[pos=0,draw,circle,fill, inner sep=0.147em]{}
					node[pos=0, below left] {$F$}
					node[pos=.49,right] {$\curve_1$};
				\addplot[domain=1.47:3.43,samples=2,->,ultra thick,black] {x}
					node[pos=.49,below right] {$\curve_2$};
				\addplot[domain=3.43:.98,samples=2,->,ultra thick,red] ({3.43},{x})
					node[pos=1.029,draw,circle,fill, inner sep=0.147em]{}
					node[pos=1.029, below right] {$\Ftilde$}
					node[pos=.49,right] {$\curve_3$};
			\end{axis}
		\end{tikzpicture}
		\vspace*{-1.96em}
	\par\end{center}
	\caption{%
	\label{fig:curveWithinSublevelsAubert}
		\small The energy $\WA$ is decreasing along $\curve_1$ and $\curve_2$ and increasing along $\curve_3$.%
	}
\end{figure}
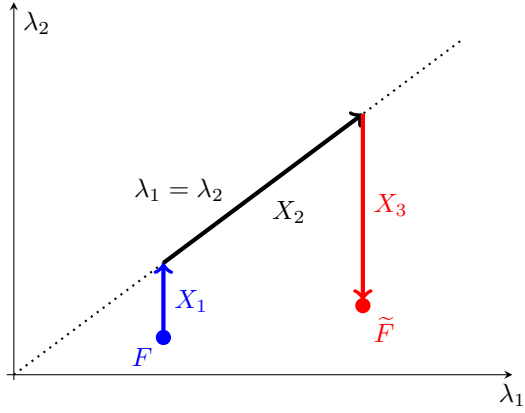

\end{appendix}

\end{document}